\documentclass[11pt]{amsart}
\usepackage[margin=0.9in]{geometry}
\usepackage{amsthm}
\usepackage{amsmath}
\usepackage{amsfonts}
\usepackage{amssymb}
\usepackage{graphicx}
\usepackage{mathrsfs}
\usepackage{graphicx,color}
\usepackage[mathscr]{eucal}
\usepackage{caption}
\usepackage{subcaption}
\usepackage{color}
\usepackage[all]{xy}
\usepackage{comment}
\usepackage{tikz-cd}
\usepackage{graphicx,color}
\usepackage{mathrsfs}
\usepackage{marvosym}
\usepackage{mathtools,tikz-cd}
\usepackage{enumitem}
\usepackage{graphicx}

\usepackage{tikz}
\usetikzlibrary{shapes.geometric, arrows}
\tikzstyle{startstop} = [rectangle, rounded corners, minimum width=2.5cm, minimum height=0.5cm,text centered, text width=2cm, draw=black, fill=white!30]
\tikzstyle{startstop2} = [rectangle, rounded corners, minimum width=2.5cm, minimum height=0.5cm,text centered, text width=2cm, draw=black, fill=white!30]
\tikzstyle{startstop3} = [rectangle, rounded corners, minimum width=2.5cm, minimum height=0.5cm,text centered, text width=2.5cm, draw=black, fill=white!30]
\tikzstyle{arrow} = [thick,->,>=stealth]

\usepackage{enumitem}

\newtheorem{theorem}{Theorem}[section]
\newtheorem{lemma}[theorem]{Lemma}

\newtheorem{corollary}[theorem]{Corollary}
\newtheorem{proposition}[theorem]{Proposition}
\newtheorem{definition}[theorem]{Definition}

\newtheorem{remark}[theorem]{Remark}
\newtheorem{conjecture}[theorem]{Conjecture}

\numberwithin{equation}{section}
\numberwithin{figure}{section}

\newcommand{\norm}[1]{\lvert\lvert#1\rvert\rvert}

\newcommand{\ZZ}{\mathbb{Z}}
\newcommand{\QQ}{\mathbb{Q}}
\newcommand{\RR}{\mathbb{R}}

\newcommand{\Mbar}{\overline{\cM}}

\newcommand{\cM}{\mathcal{M}}

\newcommand{\cC}{\mathcal{C}}

\newcommand{\cA}{\mathcal{A}}

\newcommand{\Ham}{\mathrm{Ham}}

\newcommand{\im}{\mathrm{im}}

\newcommand{\val}{\operatorname{val}}

\begin{document}

\title
[A characterization of heaviness in terms of relative symplectic cohomology]
{A characterization of heaviness in terms of relative symplectic cohomology}

\author{Cheuk Yu Mak, Yuhan Sun and Umut Varolgunes}
\date{\today}

\maketitle

\begin{abstract}
	For a compact subset $K$ of a closed symplectic manifold $(M, \omega)$, we prove that $K$ is heavy if and only if its relative symplectic cohomology over the Novikov field is non-zero. As an application we show that if two compact sets are not heavy and Poisson commuting, then their union is also not heavy. A discussion on superheaviness together with some partial results are also included.
\end{abstract}

\tableofcontents

\section{Introduction}

%A symplectic manifold $(M, \omega)$ is a smooth manifold $M$ equipped with a closed non-degenerate two-form $\omega$. 

%This extra structure enables us to study two particular types of diffeomorphisms of $M$, called symplectomorphisms and Hamiltonian diffeomorphisms. 

%Consequently, there are more rigidity phenomenons under those diffeomorphisms, which are usually characterized by intersections between subsets of $M$.

Let $(M,\omega)$ be a closed symplectic manifold. A subset $K\subset M$ is called non-displaceable from another subset $K'$ if for any Hamiltonian diffeomorphism $\phi$ we have $\phi(K)\cap K'\neq\emptyset$. If $K$ is non-displaceable from itself, then we say it is non-displaceable. Otherwise we say it is displaceable. A subset $K$ of $M$ is called stably non-displaceable if $K$ times the zero section $Z_{S^1}$ of $(T^{*}S^{1},\omega_{st})$ is non-displaceable in $M\times T^{*}S^{1}$. We refer the reader to \cite{EP09} for a detailed discussion of these notions. 
%
%\begin{definition}
%%	Many facets of displaceability \cite{EP09}:
%	\begin{enumerate}
%		\item A subset $K\subset M$ is called non-displaceable from another subset $K'$ if for any Hamiltonian diffeomorphism $\phi$ we have $\phi(K)\cap K'\neq\emptyset$. If $K$ is non-displaceable from itself, then we say it is non-displaceable. Otherwise we say it is displaceable.
%		\item Consider the cotangent bundle $T^{*}S^{1}$ with the standard symplectic structure. A subset $K$ of $M$ is called stably non-displaceable if $K$ times the zero section $Z_{S^1}$ of $T^{*}S^{1}$ is non-displaceable in $M\times T^{*}S^{1}$.
%%		\item A  subset $K\subset M$ is called strongly non-displaceable from another subset $K'$ if $\psi(K)\cap K'\neq\emptyset$ for any symplectomorphism $\psi$. If $K$ is strongly non-displaceable from itself, then we say it is strongly non-displaceable. 
%	\end{enumerate}
%\end{definition}

%All these three notions reveal, in different extent, distinctions between the symplectic topology and smooth topology of $M$. More precisely, many sets can be smoothly displaced from itself, but not by using symplectomorphisms or Hamiltonian diffeomorphisms.

As an example, note that for any $K\subset M$, the subset $K\times Z_{S^1}\subset M\times T^*S^1$ is displaceable from itself by a symplectic isotopy. On the other hand there are many interesting examples of stably non-displaceable subsets, e.g. $K=M$.

To detect obstructions in displaceability questions that involve some interesting rigidity, Floer theory provides powerful tools. For example, Lagrangian Floer cohomology is very effective in finding obstructions to displacing a Lagrangian submanifold from another one. For general compact subsets, Entov-Polterovich \cite{EP06, EP09} introduced the symplectomorphism invariant notions of heaviness and superheaviness, which have the following properties.
%\footnote{We are restricting ourselves to the idempotent $1$ in the introduction for simplicity.}. 

\begin{theorem}[Theorem 1.4 \cite{EP09}]\label{t: EP1}
	For a compact subset $K$ of $M$:
	\begin{enumerate}
		\item If $K$ is heavy, then it is stably non-displaceable from itself.
		\item If $K$ is superheavy, then $K$ intersects any heavy set.
		\item Superheavy sets are heavy, but, in general, not vice versa. 
	\end{enumerate}
\end{theorem}

The definition of heavy and superheavy subsets is based on spectral invariants \cite{Oh} in Hamiltonian Floer theory (see Section \ref{ss-heavy} for the definition in our conventions). 

%For a given compact set $K$, if it is a (good) Lagrangian submanifold, then one can use its Lagrangian Floer cohomology to verify that it is heavy or superheavy. For a general $K$, it seems there are fewer ways to do that, besides directly checking the definition.

\begin{remark}\label{rem-idem}
For experts we note that throughout the introduction we only use the unit in quantum cohomology as our idempotent for simplicity, but in the main body of the paper our results are stated (and hold) for all idempotents uniformly. 
\end{remark}

Relative symplectic cohomology \cite{Var, Var2021} of compact subsets arose as another alternative to deal with displaceability questions. In particular, the third author and Tonkonog suggested the following definitions \cite{TVar, BSV}, which are also symplectomorphism invariant.

\begin{definition}\label{def: visible}
	Let $K$ be a compact set in $M$ and let $SH_{M}(K; \Lambda)$ be the relative symplectic cohomology of $K$ in $M$ over the Novikov field $\Lambda$ (see \eqref{eq:Novi}).
	\begin{enumerate}
		\item $K$ is called SH-visible if $SH_{M}(K; \Lambda)\neq 0$, otherwise it is called SH-invisible.
		\item $K$ is called SH-full if every compact set contained in $M-K$ is SH-invisible.
%		\item $K$ is called nearly SH-visible if any compact domain containing $K$ in its interior is SH-visible.
	\end{enumerate}
\end{definition}

The motivation was that if $SH_{M}(K; \Lambda)\neq 0$, then $K$ is not stably displaceable from itself \cite{Var}. In \cite{TVar}, the following conjecture was implicit:

\begin{conjecture}\label{main-conjecture}
	A compact subset of a closed symplectic manifold is heavy if and only if it is SH-visible.
\end{conjecture}

We confirm this conjecture in this paper. The definition of SH-full was meant to be an analogue of superheaviness inspired by the fact that superheavy sets intersect heavy sets. The confirmation of Conjecture \ref{main-conjecture} shows that in fact SH-full sets are characterized as the subsets that intersect all closed heavy sets. It is plausible to believe that this is a weaker condition than superheaviness.
%
%In , several properties of SH-visible and SH-full sets are proved.
%\begin{enumerate}
%	\item If $K$ is SH-visible, then $K$ is not stably displaceable from itself.
%	\item If $K$ is SH-full, then $K$ is not strongly displaceable from any SH-visible sets.
%%	\item If $K$ is SH-full, then it is nearly SH-visible. In the other direction, a (nearly) SH-visible set may not be SH-full.
%\end{enumerate}
%One can directly see the similarity between these properties and those in Theorem \ref{t: EP1}. 

%Recall that $SH_{M}(K; \Lambda)$ is by definition obtained by base change from $SH_{M}(K)$, which is defined over the Novikov ring $\Lambda_{\geq 0}$.  $SH_{M}(K)$ contains significant quantitative information in its torsion part but this information is ignored for the purposes of the present paper. 

\begin{remark}
	$K$ is called nearly SH-visible if any compact domain containing $K$ in its interior is SH-visible. By the unitality of the restriction maps, an SH-visible set is always nearly SH-visible. 
A priori it wasn't clear that these two notions were different, see footnote 7 in \cite{BSV}. However, in this paper we will prove they coincide (see Corollary \ref{co: nearly visible}).
%But it is not clear that these two notions are strictly different, see footnote 7 in \cite{BSV}. Actually, we will prove they are the same as a by-product of the relationship with heavy sets (see Corollary \ref{co: nearly visible}). 
\end{remark}

More recently, Dickstein-Ganor-Polterovich-Zapolsky \cite{DGPZ} introduced their quantum cohomology ideal-valued quasi-measures. For example, a compact set has full measure if and only if it is SH-full (see \cite[Remark 1.46]{DGPZ}). They also defined a new notion called SH-heavy and conjectured:

\begin{conjecture}[Conjecture 1.52 \cite{DGPZ}]\label{con: DGPZ}
	A compact subset of a closed symplectic manifold is heavy if and only if it is SH-heavy.
\end{conjecture}

Here is a summary of our main results.

\begin{theorem}\label{t: main1}
	Let $(M, \omega)$ be a closed symplectic manifold. For any compact subset $K$ of $M$, the following statements are true.
	\begin{enumerate}[label=(\Alph*)]
		\item (Corollary \ref{co: heavy1}, Theorem \ref{t: heavy2}) $K$ is SH-visible if and only if $K$ is heavy. 
		\item (Proposition \ref{p:basicheavy}) If $K$ is SH-heavy then $K$ is heavy.
		\item (Corollary \ref{co: superheavy1}) If $K$ is superheavy then $K$ is SH-full.
	\end{enumerate}
\end{theorem}

The proof of our most fundamental result, that $SH$-visible implies heavy, relies on a chain level product structure, which was recently constructed in \cite{AGV}. We only use a small part of the structure constructed in that paper.

%\begin{figure}
%	\begin{tikzpicture}[xscale=0.8,yscale=0.8]
%		\node at (-7,5) {SH-visible};
%		\node at (0,5) {heavy};
%		\node at (7,5) {nearly SH-visible};
%		\node at (0,0) {SH-heavy};
%		\draw [->] (-5.5,5.2)--(-1,5.2);
%		\node [above] at (-3,5.2) {$(A)$};
%		\draw [->] (-1,4.8)--(-5.5,4.8);
%		\node [below] at (-3,4.8) {$(A)$};		
%		\draw [->] (1,5.2)--(5,5.2);
%		\node [above] at (3,5.2) {$(A)$};
%		\draw [->] (5,4.8)--(1,4.8);
%		\node [below] at (3,4.8) {$(A)$};	
%		%\draw [->] [dashed] (-0.3,4.5)--(-0.3,0.5);
%		%\node [left] at (-0.5,2.4) {$(E)$};
%		\draw [->] (0.3,0.5)--(0.3,4.5);
%		\node [right] at (0.5,2.4) {$(B)$};
%		\node at (-5,-2) {SH-full};
%		\node at (5,-2) {superheavy};
%		\draw [->] [dashed] (-4,-1.8)--(3.5,-1.8);
%		\draw [->] (3.5,-2.2)--(-4,-2.2);
%		\node [below] at (0,-2.2) {$(C)$};
%		\node [above] at (0,-1.8) {$(D)$};
%		
%		
%	\end{tikzpicture}
%	\caption{Summary of relations.}\label{Relation}
%\end{figure}

%Figure \ref{Relation} gives a summary of relations among various notions of rigidity. Solid lines represent general implications. We also have a result (D) under some conditions, which will be explained later.

Here is an interesting corollary of Theorem \ref{t: main1} part (A) and the Mayer-Vietoris property of relative symplectic cohomology:

\begin{corollary}\label{cor-union-heavy}
Let $K_1$ and $K_2$ be Poisson commuting compact subsets of $M$, see Definition \ref{def: Poisson}. If they are both not heavy, then neither is their union $K_1\cup K_2$. 

\end{corollary}
\begin{proof}
	Let $K_{1}, K_{2}$ be two non-heavy sets. Then Theorem \ref{t: main1} part (A) says that $SH_{M}(K_{1}; \Lambda)= SH_{M}(K_{2}; \Lambda)=0$ and thus, by the unitality of restrictions, also $SH_M(K_1\cap K_2;\Lambda) = 0$. Moreover, if $K_{1}, K_{2}$ are Poisson commuting, the Mayer-Vietoris property of relative symplectic cohomology shows that $SH_{M}(K_{1}\cup K_{2}; \Lambda)=0$. 
\end{proof}

An important example is when $K_1$ and $K_2$ are disjoint. As far as we know, results of this form were only proved for very special $M$ and $K_i$, for example, when $M$ is aspherical and $K_i$ are incompressible Liouville domains \cite{HLS,T}.
% even this is new.

We also give a positive answer to \cite[Question 4.9]{Entov} using Corollary \ref{cor-union-heavy} and that $SH_M(M;\Lambda)\simeq H^*(M;\Lambda)\neq 0$.
\begin{corollary}\label{c:involution}
Any involutive map $\Phi: M\to \mathbb{R}^N$ has at least one heavy fiber.
\end{corollary}

This corollary in turn implies a positive answer to \cite[Question 3.4]{Entov} about dispersion-freeness of symplectic quasi-states constructed using the homogenized spectral invariant, see \cite[Theorem 3.1]{Entov}. For a more in depth discussion of this question see \cite[Section 6.2]{Pol}. This implication was communicated to us by Adi Dickstein, Yaniv Ganor, Leonid Polterovich and Frol Zapolsky after the first preprint version of the paper was released. It would take us too far to try to introduce the relevant definitions or notions regarding symplectic quasi-states here, so we refer the reader to \cite{Entov} for those. As noted in Remark \ref{rem-idem}, the results stated in the introduction are in fact proved for all idempotents (not just for the unit) and we use this below.

\begin{corollary}
Every symplectic quasi-state constructed in \cite[Theorem 3.1]{Entov} corresponding to the unit in a field factor of $QH(M;\Lambda)$  is dispersion-free.
\end{corollary}

\begin{proof}
A quasi-state $z$ is dispersion-free if and only if it satisfies $z(f^2)=z(f)^2$ for each continuous function $f$.

Let $z$ be the symplectic  quasi-state corresponding to the unit in a field factor of $QH(M;\Lambda)$.
By the $C^0$-Lipschitz property of $z$ (cf. Theorem \ref{t:property}(4)), it is enough to show that $z(f^2)=z(f)^2$ for each smooth function $f$. Consider $f$ as an involutive map $f: M \to \mathbb{R}$. Corollary \ref{c:involution} implies that this map has a heavy fiber, say at the point $t \in \mathbb{R}$. Since $z$ is a genuine quasi-state, this fiber is also superheavy by \cite[Theorem 4.1.e.]{Entov}. It then follows that the pushforward functional $f_*z$ is in fact the evaluation at $t$ directly from the definition of heavy and superheavy (see e.g. Definition \ref{def-heavy}). In particular $z(f^2)$ is the evaluation of the function $x \mapsto x^2$ at $t$, that is just $t^2: z(f^2)=t^2$. But $z(f) = t$, that is $z(f^2)=t^2=z(f)^2$. It follows that $z$ is dispersion-free.
\end{proof}

Another immediate corollary is a solution to Conjecture 1.3 in \cite{TVar}.

\begin{corollary}
If $L$ is a Lagrangian submanifold which admits a bounding cochain with non-zero Floer cohomology, then $L$ is not contained in an $SH$-invisible set $K$.  
\end{corollary}

\begin{proof}
By Theorem 1.6 in \cite{FOOO}, $L$ is heavy so Theorem \ref{t: main1} part (A) implies that $SH_M(L; \Lambda) \neq 0$.
\end{proof}

\begin{remark}
If $K$ is a union of Lagrangian submanifolds $\{K_i\}$ one can also observe certain parallels between the notion of (split)-generating the Fukaya category and the notion of superheaviness/SH-fullness. Yash Deshmukh presented to us a convincing (and not too difficult) argument showing that if $\{K_i\}$ satisfies the Abouzaid generation criterion then $K=\cup K_i$ is SH-full.
\end{remark}

We note the following partial converse to part (C) of Theorem \ref{t: main1}.

%\footnote{There was a (not clean) implication $(E)$ in previous versions, see Theorem 1.8 in the 1016 version. It is a three-line corollary under certain conditions. It depends on whether we want to put it here.}

\begin{theorem}\label{t:superheavyc1}
Let $K$ be a compact subset of $M$, and let $K_0\supset K_1\supset \ldots$ be compact subsets that all contain $K$ in their interior such that $\bigcap K_i=K$.
If $c_{1}(TM)$ vanishes on $\pi_2(M)$ and the $\mathbb{Z}$-graded relative symplectic cochain complexes $SC^*_M(K_n;\Lambda)$ have finite boundary depth in all degrees for all $n\geq 1$, then $K$ being SH-full implies that it is superheavy.
\end{theorem}
\begin{remark}
Our methods would extend to show a similar result in certain negative and positive monotone cases as well  (see Proposition \ref{p:PSS=0} and the discussion afterwards) but we are not confident of their usefulness and optimality. Hence we refrain from a detailed discussion.
\end{remark}

For example, as a corollary, the skeleta from Theorem 1.24 of \cite{TVar} can now be proved superheavy. We omit a full proof of this in order not to introduce notation that is not used elsewhere in the paper.

We obtain Theorem \ref{t:superheavyc1} from a more general result Theorem \ref{t:SHred_superheavy} that involves a condition involving reduced symplectic cohomology $SH^{red}_M(M;\Lambda)$. This can be defined as the quotient of $SH_M(M;\Lambda)$ by its valuation $\infty$ subspace (see the discussion surrounding \eqref{eq:SHSHred}). 

Finally, let us note the following result, which is a by-product of our discussion:

\begin{proposition}\label{prop- red-vanishing}$SH^{red}_M(K;\Lambda)=0$ implies $SH_M(K;\Lambda)=0.$
\end{proposition} 

The proof is given in Section \ref{ss-superheavy} after all the necessary notions are properly introduced.

\subsection*{Acknowledgements}
We thank Leonid Polterovich and Sobhan Seyfaddini for helpful discussions and their interest. We also thank the anonymous referee for helpful suggestions.

C.M. was supported by the Simons Collaboration on Homological Mirror Symmetry \#652236 and the Royal Society University Research Fellowship. U.V. was supported by the T\"{U}B\.{I}TAK 2236 (CoCirc2) programme with a grant numbered 121C034.

\section{Preliminary}
First we specify the ring and field that will be used. The Novikov field $\Lambda$ is defined by
\begin{align}\label{eq:Novi}
\Lambda=\left \lbrace \sum_{i=0}^{\infty}a_{i}T^{\lambda_{i}}\mid a_{i}\in \mathbb{Q}, \lambda_{i}\in\mathbb{R}, \lambda_{i}<\lambda_{i+1}, \lim_{i\rightarrow \infty}\lambda_{i}=+\infty \right \rbrace
\end{align}
where $T$ is a formal variable. There is a valuation $\val: \Lambda\to \RR \cup \lbrace +\infty\rbrace$ given by $\val (\sum_{i=0}^{\infty}a_{i}T^{\lambda_{i}}):= \min_{i}\lbrace \lambda_{i} \lvert a_{i}\neq 0\rbrace$ and $\val (0):= +\infty$. Then for any real number $r$, we define $\Lambda_{> r}$ be the subset of elements with valuation greater than $r$. Similarly we define $\Lambda_{\ge r}$. In particular, $\Lambda_{\ge 0}$ is called the Novikov ring. It is a commutative ring with a unit $1\in \QQ$. Both $\Lambda_{> r}$ and $\Lambda_{\ge r}$ are modules over $\Lambda_{\ge 0}$. When $r\ge 0$, they are ideals of $\Lambda_{\ge 0}$.

On the level of vector spaces, we identify the quantum cohomology $QH^{*}(M; \Lambda)$ with $H^{*}(M; \Lambda)= H^{*}(M; \QQ)\otimes_{\QQ}\Lambda$. Three (co)homology classes that will be frequently used are
\begin{enumerate}
	\item $1\in H^{0}(M; \Lambda)$, the unit class;
	\item $[vol]\in H^{2n}(M; \Lambda)$, the volume class;
	\item $[M]\in H_{2n}(M; \Lambda)$, the fundamental class,
\end{enumerate}  
where $\langle [vol], [M]\rangle=1$ under the cohomology-homology pairing.

Another related $\mathbb{Z}$-graded algebra $\Lambda_{\omega}=\bigoplus_{n\in\mathbb{Z}} \Lambda_{\omega}^n$ is defined by
$$
\Lambda_{\omega}^n :=\left \lbrace \sum_{i=0}^{\infty}a_{i}e^{A_{i}}\mid a_{i}\in \mathbb{Q}, A_{i}\in \pi_{2}^n(M)/\sim, \lim_{i\rightarrow \infty}\omega(A_{i})=+\infty \right \rbrace
$$
where $\pi_{2}^n(M)$ is the subset of $\pi_2(M)$ of classes satisfying $2c_1(M)[A]=n$ and $\sim$ is the relation that 
$$
A\sim A' \quad \text{iff} \quad \omega(A)=\omega(A').
$$
Note that we have an algebra map $\Lambda_{\omega}\to \Lambda$ by $e^{A}\mapsto T^{\omega(A)}$. 

%Unless $c_1(M)$ vanishes on $\pi_2(M)$, ${\Lambda}$ is not a graded module over ${\Lambda_{\omega}}$. 

%For the free $\Lambda_{\omega}$-module $I_{\Lambda_{\omega}}$ and the free $\Lambda$-module $I_{\Lambda}$ generated by the same set $S$ there is a base-change map
%$$
%I_{\Lambda_{\omega}}\otimes_{\Lambda_{\omega}}\Lambda \to I_{\Lambda}: \quad a\cdot e^{A}\otimes T^{\lambda} \to a\cdot T^{\lambda+ \omega(A)}, \text{ for every }a\in S.
%$$

\begin{remark}
It is instructive to construct the map $\Lambda_{\omega}\to \Lambda$ in stages: \begin{enumerate}
 \item forget the grading.
 \item take the quotient by the subalgebra given by the group-algebra of $\{\omega(A)=0\}\subset\pi_2(M).$
\item make a completion to come to the completed group-algebra of the abelian group $\omega(\pi_2(M)).$
\item use the obvious inclusion $\omega(\pi_2(M))\subset\mathbb{R}.$
\end{enumerate}
\end{remark}

\begin{remark}
	We defined our Novikov rings with ground field being $\QQ$. But all our theorems work for any ground field containing $\QQ$. Moreover, if the full Hamiltonian Floer theory package on $M$ can be defined over some other commutative ring, then our results work for those as well.
\end{remark}

\subsection{Spectral invariants}

From the outset let us note that we need to use virtual techniques to do Hamiltonian Floer theory on a general closed symplectic manifold. Here we follow the same convention in \cite{Var}, where Pardon's approach \cite{Pardon} was used. On the other hand, the analysis in this article is independent of this choice and we will not include all the necessary technical details and language regarding Pardon's implicit atlases etc. in our discussion. We will mostly omit discussing almost complex structures as well.

With this in mind, we briefly review the definition of Hamiltonian Floer groups, referring to \cite{HS} for details.\footnote{In this article we always work on closed symplectic manifolds and only use contractible orbits to define Hamiltonian Floer groups.} Let $(M, \omega)$ be a closed symplectic manifold and $H$ be a non-degenerate Hamiltonian function $M\times S^1\to\mathbb{R}$. Let $\gamma$ be a contractible one-periodic orbit of $H$ and $u$ be a disk capping of $\gamma$. The action of $(\gamma, u)$ is 
$$
\cA_{H}(\gamma, u):= \int_{\gamma} H + \int u^{*}\omega,
$$ and we can associate to it an integer degree using the Conley-Zehnder index $CZ(\gamma, u)$.
We define an equivalence relation on capped orbits by
\[
(\gamma, u) \sim  (\gamma', u') \text{ if and only if }\gamma=\gamma' \text{ and } \cA_{H}(\gamma, u)=\cA_{H}(\gamma', u') \text{ and } CZ(\gamma, u)=CZ(\gamma', u').
\]
The equivalence classes are denoted by $[\gamma, u]$.
The classical Hamiltonian Floer complex in degree $n\in \mathbb{Z}$, i.e. $CF^n(H)$ consists of all formal sums
$$
x= \sum_{i} x_{i}\cdot [\gamma_{i}, u_{i}], \quad \cA_{H}([\gamma_{i}, u_{i}])\to +\infty.
$$
Here the coefficient $x_{i}$'s are in $\QQ$, and $[\gamma_{i}, u_{i}]$'s are equivalence classes of capped orbits of degree $n$.
The valuation of such a $\mathbb{Q}$-linear combination of capped orbits is defined as the minimum of the actions of summands with nonzero coefficients. Novikov ring $\Lambda_{\omega}$ acts on $CF^*(H)$ by $e^{A}\cdot [\gamma, u]:= [\gamma, u+A]$, making $CF^*(H)$ a free graded $\Lambda_{\omega}$-module. Choosing a reference cap for each $1$-periodic orbit will give a particular basis for $CF^*(H)$.

The Floer differential $d: CF(H)\to CF(H)$ is defined by counting Floer cylinders. We remark that our conventions about Floer equations and Conley-Zehnder indices follow \cite{Var, BSV}. In particular, the differential increases the action and index of a capped orbit. The resulting homology is written as $HF^*(H)$. For any $p\in \mathbb{R}$ we define
$$
CF_{\ge p}(H):= \lbrace x\in CF(H) \mid \cA_{H}(x)\ge p\rbrace.
$$
Since the differential $d$ increases the action, we also have a homology $HF_{\ge p}^* (H):= H^*(CF_{\ge p}(H), d)$. And there is a natural map $i_{p}^*: HF_{\ge p}^* (H)\to HF^*(H)$ induced by the inclusion. 

Next let $PSS^{H}: QH^*(M; \Lambda_{\omega}) \to HF^*(H)$ be the PSS isomorphism \cite{PSS} between the quantum cohomology of $M$ and $HF(H)$. 
We usually omit the superscript $H$ when it is clear. 
%It is induced by a chain level map
%$$
%cPSS^{H}: CM(f, \Lambda_{\omega})\to CF(H),
%$$
%where $f$ is a Morse function on $M$ and the map counts spiked disks connecting critical points of $f$ and orbits of $H$. 
For any non-zero pure degree class $a\in QH^n(M; \Lambda_{\omega})$ the spectral invariant $c(a; H)$ is defined as
$$
c(a; H):= \max \lbrace p\in \RR \mid PSS^{H}(a)\in \im(i_{p}^n) \rbrace.
$$
One can check that it also equals to $\max\lbrace \cA_{H}(x)\mid x\in CF^n(H), dx=0, [x]=PSS^{H}(a)\rbrace$. 

\begin{remark}
For writing $\max$ in the definition of the spectral invariant, we are relying on the existence of best representatives result of Usher, Theorems 1.3 and 1.4 from \cite{U}.
\end{remark}

This is the setup generally used in the spectral invariant literature. We will simplify it a bit (using a simpler Novikov cover) and justify why nothing is lost by doing this. Consider the Floer complex $CF(H; \Lambda)$ which is the $\Lambda$-vector space generated by the $1$-periodic orbits (without cappings). We equip it with the $\mathbb{Z}/2$ grading coming from the Lefschetz signs of the corresponding fixed points. The cylinders contributing to the Floer differential are weighted by the topological energy. More precisely, let $u$ be a Floer cylinder connecting two orbits $\gamma_{-}, \gamma_{+}$. Then its contribution in the differential is weighted by $T^{E_{top}(u)}$, where
$$
E_{top}(u):= \int_{\gamma_{+}}H -\int_{\gamma_{-}}H +\int u^{*}\omega.
$$
Note that we have $E_{top}(u)\ge 0$. We write the resulting homology as  $HF(H; \Lambda)$. For $x=\sum x_{i}\gamma_{i}$ with $\gamma_{i}$'s being orbits and $x_{i}\in \Lambda$, we define
$$
\val(x):= \min_{i} \lbrace \val(x_{i})\rbrace.
$$

We have the map
	$$
	CF(H)\otimes_{\Lambda_{\omega}}\Lambda\to CF(H; \Lambda): \quad [\gamma, u]\otimes T^{E}\to T^{\cA_{H}([\gamma, u])+E}\cdot \gamma .
	$$
	It is an isomorphism of chain complexes, which also preserves the valuations.

We have a $PSS$-map over $\Lambda$ given as the composition:
$$
PSS^{H}_{\Lambda}: QH(M; \Lambda)\to HF(H)\otimes_{\Lambda_{\omega}}\Lambda\to HF(H; \Lambda).
$$ 
Similarly for any non-zero class $a\in QH(M;\Lambda)$, we define
$$
c_{\Lambda}(a; H):= \max \lbrace \val(x)\in \RR \mid x\in CF(H; \Lambda), dx=0, [x]=PSS^{H}_{\Lambda}(a)\rbrace.
$$
And a handy comparison result is the following.

\begin{proposition}\label{p:compareSpectral}
	Let $H$ be a non-degenerate function and $a\in QH(M; \Lambda_{\omega})$ be a non-zero pure integral degree class.
Let $\tilde{a} \in QH(M; \Lambda)$ be the image of $a$ under the map $QH(M; \Lambda_{\omega}) \to QH(M; \Lambda)$ induced by $\Lambda_{\omega}\to \Lambda$. Then we have that $c(a; H)= c_{\Lambda}(\tilde{a}; H)$.
\end{proposition}
\begin{proof}
%	Since $CF(H)$ is a $\Lambda_{\omega}$-module, we can consider $CF(H)\otimes_{\Lambda_{\omega}}\Lambda$ modulo the following relation
%	$$
%	[\gamma, u]\sim [\gamma', u']T^{E} \quad \text{iff} \quad \gamma=\gamma', \int u^{*}\omega = E+ \int u'^{*}\omega.
%	$$
%	A new element $[\gamma, u]T^{E}$ has an action $\cA_{H}^{T}([\gamma, u]):=\cA_{H}([\gamma, u])+E$.

We can construct a chain level PSS map
$$
cPSS^{H}: CM(f, \Lambda_{\omega})\to CF(H),
$$
where $f$ is a Morse function on $M$ and the map counts spiked disks connecting critical points of $f$ and orbits of $H$. By base changing and using the action scaling map above we also obtain $
cPSS^{H}_{\Lambda}: CM(f, \Lambda)\to CF(H;\Lambda),
$ so that we have a commutative square of chain maps
	\[\begin{tikzcd}
		CM(f; \Lambda_{\omega}) \arrow{r}{cPSS^H} \arrow[swap]{d} & CF(H) \arrow{d} \\
		 CM(f; \Lambda) \arrow{r}{cPSS^H_{\Lambda}} & CF(H; \Lambda)
	\end{tikzcd}
	\]

Trivially $c_{\Lambda}(\tilde{a}; H)\geq c(a; H)$. For the opposite direction note that for all $n\in \mathbb{Z}$, the maps  $CF^n(H)\to  CF(H;\Lambda)$ are valuation preserving (and hence injective) and $HF^n(H)\to  HF(H;\Lambda)$ are injective. 
	
	 We consider a simple case to illustrate the proof. Suppose that $x=T^{E}\gamma \in CF(H; \Lambda)$ represents $PSS^H_{\Lambda}(\tilde{a})$ and $PSS^H(a)$ is represented by $y$ for a capped orbit $y=[\gamma_{1}, u_{1}]$. By the commutativity of the above square, we have that $T^{E}\gamma $ is homologous to $y_T=T^{\cA_{H}([\gamma_{1}, u_{1}])}\gamma_{1}$ in $CF(H; \Lambda)$. Assume that
	$$
	T^{\cA_{H}([\gamma_{1}, u_{1}])}\gamma_{1}- T^{E}\gamma = dz, \text{ with }z=T^{E_{2}}\gamma_{2}
	$$
	for some $\gamma_{2}$. Then we must have rigid Floer cylinders $u_{2}$ from $\gamma_{2}$ to $\gamma_{1}$, and $u_{3}$ from $\gamma_{2}$ to $\gamma$. We have
	$$
	E_{2}+ E_{top}(u_{2})= \cA_{H}([\gamma_{1}, u_{1}]), \quad E_{2}+ E_{top}(u_{3})= E.
	$$
	We can directly check that $[\gamma, u_{3}-u_{2}+u_{1}]$ is a capped orbit of the same degree as $[\gamma_{1}, u_{1}]$ which represents $PSS^H(a)$ and with action $E$.
	
	We now give the proof. Since $H$ is non-degenerate, we have finitely many one-periodic orbits $\lbrace \gamma_{j}\rbrace_{j=1, \cdots, N}$ in the mod $2$ degree of $a$, that is $\deg(\tilde{a}),$  and $\lbrace \beta_{k}\rbrace_{k=1, \cdots, L}$ in degree $\deg(\tilde{a})-1$. 
%	Note that we can easily characterize the image of $CF(H)\to  CF(H;\Lambda)$. It is a direct sum of $\mathbb{Q}$-subspaces $$\{T^E\gamma_j\mid E= \cA_{H}([\gamma_{j}, u]) \text{ mod }\omega(\pi_2(M))\}$$ and $$\{T^E\beta_k\mid E= \cA_{H}([\beta_k, u]) \text{ mod }\omega(\pi_2(M))\}.$$

	 Let 
	$$
	x= \sum_{j=1}^{N} \sum_{l} x^{l}_{j} T^{E^{l}_{j}} \gamma_{j}, \quad x^{l}_{j}\in \QQ
	$$
	be an element representing $PSS^H_{\Lambda}(\tilde{a})$ with valuation $c_{\Lambda}(\tilde{a}; H)$, and
	$$
	y= \sum_{j=1}^{N}\sum_{i}y_{j}^{i} [\gamma_{j}, u^{i}_{j}], \quad y_{j}^{i}\in \QQ, \quad \lim_{i\to +\infty} \cA_{H}([\gamma_{j}, u_{j}^{i}])= +\infty
	$$
	be a pure degree element representing $PSS^H(a)$. After the base change we get 
	$$
	y_{T}:= \sum_{j=1}^{N}\gamma_{j}(\sum_{i}y_{j}^{i}T^{\cA_{H}([\gamma_{j}, u_{j}^{i}])})
	$$
	such that $x$ and $y_{T}$ are homologous.
	That is, we have another element $z$ with $dz= y_{T}-x$. We write 
	$$
	z= \sum_{k=1}^{L}\sum_{m}z^{m}_{k}T^{E^{m}_{k}}\beta_{k}, \quad z^{m}_{k}\in \QQ.
	$$
	
%	
%	If  $E^{l}_{j}= \cA_{H}([\gamma_{j}, u]) \text{ mod }\omega(\pi_2(M))$ for all $j$ and $l$, we are done. Assume that we are not.
	
%	 Write $y_{T}-x$ as a $\mathbb{Q}$-linear (possibly infinite) combination of $T^E\gamma_j$. 
	 
	 Let $n$ be the degree of $a$. If $rT^E\gamma_j$ with $r\in \mathbb{Q}^*$ satisfies $$E= \cA_{H}([\gamma_{j}, u])$$ for some index $n$ cap $u$ of $\gamma_j$, we call it good; otherwise we call it bad. 
	
	Consider the $\mathbb{Q}$-linear expression of $d(T^E\beta_k)$. The argument in the simple case above shows that either all of its terms are good or they are all bad. Let us call $rT^E\beta_k$ with $r\in \mathbb{Q}^*$ good or bad accordingly as well. Define $z'$ by removing from the $\mathbb{Q}$-linear expression of $z$ all the bad terms. Therefore, $dz'$ is a $\mathbb{Q}$-linear combination of (possibly infinitely many) good terms.
 We define $x'$ by $dz'= y_T-x'.$

	Clearly $x'$ is in the image of $CF^n(H)\to  CF(H;\Lambda)$ (say of $\tilde{y}$) and it represents $PSS_\Lambda(\tilde{a})$. It follows that the valuation of $\tilde{y}$ is at least as large as that of $x$ because the terms in the $\mathbb{Q}$-linear expansion of $x'$ is simply a subset of the terms of $x$. Therefore the valuation of $\tilde{y}$ and hence $c(a; H)$ is at least as large as $c_{\Lambda}(\tilde{a}; H).$

%	Suppose that with respect to the bases $\lbrace \gamma_{j}\rbrace, \lbrace \beta_{k}\rbrace$, the Floer differential is
%	$$
%	d\beta_{k}= \sum_{j, s_{j}} c^{j,k}_{s_{j}}\cdot T^{E_{top}(s_{j})}\cdot \gamma_{j}, \quad c^{j,k}_{s_{j}}\in \QQ.
%	$$
%	Here $s_{j}$ represent the homotopy classes of Floer cylinders from $\beta_{k}$ to $\gamma_{j}$. 
%	
%	
%----------------------	
%	
%	Then for each summand $x^{l}_{j}T^{E^{l}_{j}}\gamma_{j}$ of $x$ we can associate it with a capped orbit
%	$$
%	\sum_{i}\sum_{j'}\sum_{k}\sum_{m} y^{i}_{j}\cdot x^{l}_{j}\cdot z^{m}_{k}\cdot c^{j,k}_{s_{j}}\cdot c^{j',k}_{s_{j'}}\cdot [\gamma_{j}, u^{i}_{j}- v^{j',k}_{s_{j'}}+ v^{j,k}_{s_{j}}],
%	$$
%	where the summation is over the following conditions
%	$$
%	E^{m}_{k}+ E_{top}(v^{j',k}_{s_{j'}})= \cA_{H}([\gamma_{j'}, u_{j'}]), \quad E^{m}_{k}+ E_{top}(v^{j,k}_{s_{j}})= T^{E^{l}_{j}}.
%	$$
%	Finally add up all those capped orbits we get a capped orbit which represents $PSS(a)$ and with action the same as $\val(x)$.\footnote{should have a better proof here.}
%	 
\end{proof}

In the rest of this article we will use the second definition of spectral invariant, but with the first notation $c(a; H)$. It has lots of good properties, which we list here and refer to \cite{Oh, FOOO, EP09} for proofs and more discussions. 

\begin{theorem}\label{t:property}
	Let $a$ be a non-zero class in $QH(M; \Lambda)$, and let $H$ be a non-degenerate Hamiltonian function. The number $c(a; H)$ has the following properties.
	\begin{enumerate}
		\item (Finiteness) $c(a; H)$ is a finite number.
		\item (Valuation shift) $c(xa;H)$= $c(a; H)+ \val(x)$ for any non-zero $x\in\Lambda$.
		\item (Hamiltonian shift) $c(a; H+\lambda(t))= c(a; H)+ \int_{S^{1}}\lambda(t)$ for any function $\lambda(t)$ on $S^{1}$.
		\item (Lipschitz property) $\int_{S^{1}}\min_{M} (H_{1}- H_{2})\leq c(a; H_{1})- c(a; H_{2})\leq \int_{S^{1}} \max_{M} (H_{1}- H_{2})$. In particular, if $H_{1}\geq H_{2}$ then $c(a; H_{1})\geq c(a; H_{2})$.
		\item (Triangle inequality) $c(a_{1}*a_{2}; H_{1}\# H_{2})\geq c(a_{1}; H_{1})+ c(a_{2}; H_{2})$ where $*$ is the quantum product.
		\item (Homotopy invariance) $c(a; H_{1})= c(a; H_{2})$ for any two normalized Hamiltonian functions generating the same $\phi \in \widetilde{\Ham}(M)$.
	\end{enumerate}
\hfill \qedsymbol{}
\end{theorem}

By the Lipschitz property, we can extend the definition of spectral numbers to all continuous functions $M\times S^1\to \mathbb{R}$ by $C^0$ approximation. 

\begin{remark}\label{rem:low-semi-spec} Let us also note that by monotone approximation one can define spectral numbers for lower (or upper) semi-continuous functions $M\to \mathbb{R}\cup\{\pm\infty\}$. An important example of a lower semi-continuous function is 
$$
\chi_K(x)=\begin{cases}
			0, & \text{if $x\in K$}\\
            \infty, & \text{otherwise}
		 \end{cases}
$$ for $K$ a closed subset.\end{remark}

Next we recall a Poincar\'e duality property. For two classes $a,b\in QH(M; \Lambda)$ we define
\begin{equation}\label{eq: pairing}
	\Pi(a, b):= \pi\langle a*b, [M]\rangle.
\end{equation}
Here $*$ denotes the quantum product, $\langle\cdot, \cdot\rangle:  H^{*}(M; \Lambda)\times H_{*}(M; \Lambda)\to \Lambda$ is the pairing between cohomology and homology, and $\pi: \Lambda\to \QQ$ is the projection to the constant term.

\begin{lemma}[\cite{EP03}, Lemma 2.2]\label{l:Poincare}
For any $b \in QH(M; \Lambda) \setminus \{0\}$ and smooth $H:M\times S^1\to \mathbb{R}$, we have
\begin{align}\label{eq:PoincareC}
c(b;H)=-\sup\{c(a;\bar{H})| \Pi(a,b) \neq 0\}.
\end{align}
where the Hamiltonian $\bar{H}(x,t):=-H(\phi_H^t(x),t)$ generates $(\phi_H^{t})^{-1}$. 
\hfill \qedsymbol{}
\end{lemma}

The identity \eqref{eq:PoincareC} has the same mathematical content as \cite[Lemma 2.2]{EP03} but is in our sign conventions.

%\begin{corollary}\label{c:Poincare}
%We always have $c([vol_M],H)=- c(1,\bar{H})$.[WRONG]
%\end{corollary}

Given a time-independent smooth function $H$ on $M$ and a non-zero idempotent $a\in QH(M; \Lambda)$, the homogenized spectral invariant is defined as
$$
\mu(a;H):= \lim_{n\to +\infty} \frac{c(a;nH)}{n}.
$$
This limit is well-defined and has several properties. See \cite{EP06}. We note
\begin{equation}\label{l:basicineq}
c(a;H) \le \mu(a;H),
\end{equation} which is immediate from the triangle inequality.

%\begin{lemma}\label{l:basicineq}
%We always have
%\[
%c(1;H) \le \mu(1;H) \le \mu(a;H) 
%\]
%for any non-zero idempotent $a \in QH(M; \Lambda)$.
%%\[
%%c(1;H) \le \mu(1;H) \le \mu([vol_M];H) \le c([vol_M];H)
%%\]
%\end{lemma}
%
%\begin{proof}
%The first inequality follows from the triangle inequality $c(1;H \# H) \ge c(1;H) +c(1;H)$ for all $H$.
%The second inequality follows from the triangle inequality $c(a;H) \ge c(1;H) +c(a;0)=c(1;H)$ for all $H$.
%%The last inequality follows from the triangle inequality (of coproduct) $c([vol_M];H\#H) \le c([vol_M];H) +c([vol_M];H)$ for all $H$.
%\end{proof}

\subsection{Heaviness and superheaviness}\label{ss-heavy}
Let $a\in QH(M; \Lambda)$ denote a non-zero idempotent throughout this section.
\begin{definition}\label{def-heavy}
A compact subset $K \subset (M, \omega)$ is $a$-heavy if $\mu(a;H) \le \max_{K} H$ for all $H \in C^{\infty}( M)$.
It is $a$-superheavy if $\mu(a;H) \ge \min_{K} H$ for all $H \in C^{\infty}(M)$. When $a$ is the unit, we often omit it from the notation.
\end{definition}

\begin{remark}
The original definition of (super)heavy sets in \cite{EP09} uses a homological version of spectral invariants. And there is a duality formula (Section 4.2 in \cite{LZ}) relating them with our cohomological spectral invariants. Sticking to the unit case for simplicity, we have
$$
c(1; H)= -c_{ho}([M]; -H), \quad \text{hence} \quad \mu(1; H)= -\mu_{ho}([M]; -H).
$$
The definition in \cite{EP09} of heaviness is $\mu_{ho}([M];H) \ge \min_K H$ for all $H$, which is equivalent to $\mu(1;H)=-\mu_{ho}([M]; -H) \le - \min_K (-H)=\max_{K} H$ and hence our definition above. The translation for superheaviness is similar. 

%\footnote{In the symplectic ideal value measure paper, they use that $\mu([M],H)=\mu(vol_M,H)$, which is not true if $M$ is not aspherical}
\end{remark}

%\begin{lemma}
%A compact subset $K \subset (M, \omega)$ is heavy if and only if $\mu([vol_M],H) \ge  \min_{K} H$ for all $H \in C^{\infty}( M)$.
%It is superheavy if and only if $\mu(vol_M,H) \le \max_{K} H$ for all $H \in C^{\infty}(M)$.
%\end{lemma}

%\begin{proof}
%By Corollary \ref{c:Poincare}, we have $c([vol_M],H)=- c(1,\bar{H})$.
%Since $\max_{K} H =- \min_{K} \bar{H}$, the result follows.

%\end{proof}

\begin{lemma}[$\mu$-heaviness and $c$-heaviness]\label{l:equivheavy}
A compact subset $K \subset (M, \omega)$ is $a$-heavy if and only if $c(a; H) \le \max_{K} H$ for all $H \in C^{\infty}( M)$.
\end{lemma}

\begin{proof}
%It suffices to consider the spectral invariant with respect to the identity element.
By the inequality \eqref{l:basicineq}, we have $\mu(a;H) \ge c(a;H)$.
Therefore, if $\mu(a;H) \le \max_{K} H$, then $c(a;H) \le \max_{K} H$.

On the other hand, if $c(a;H) \le \max_{K} H$ for all $H$, then we have $c(a;nH) \le \max_{K} nH=n \max_K H$ for all $H$.
This implies that $\mu(a;H) \le \max_{K} H$ for all $H$.
\end{proof}

\begin{remark}
We cannot use $c(a;H)$ instead of $\mu(a;H)$ to define superheaviness. Consider the unit sphere $S^2\subset \mathbb{R}^3$ with its standard area form and let $z:S^2\to [-1,1]$ be the height function. Let $K\subset S^2$ be the upper hemisphere defined by $z\geq 0$. It is well-known that $K$ is superheavy. Consider the non-degenerate autonomous Hamiltonian $H:=\epsilon z$ with $0<\epsilon \ll 1$. By superheaviness, we have $\mu(1;H) \ge \min_{K} H=0,$ but $$c(1;H)=-\epsilon < \min_{K} H.$$
%More precisely, $\mu(1,H) \ge \min_{K} H$ does not imply $c(1,H) \ge \min_{K} H$.
%%Equivalently even if $\mu(vol_M,H) \le \max_{K} H$, it does not mean that $c(vol_M,H) \le \max_{K} H$.
%Indeed, in $S^2$, the spectral invariant of a Hamiltonian supported in a displaceable disk does not have to be $0$ but the homogenized spectral invariant is $0$.
\end{remark}

Now we recall two criteria for heaviness and superheaviness.

\begin{lemma}\label{l: criteria}
	Let $K$ be a compact set in $M$.
	\begin{enumerate}
		\item $K$ is $a$-heavy if and only if for any non-negative function $H$ on $M$ with $H\lvert_{K}=0$, we have that $\mu(a; H)=0$;
		\item $K$ is $a$-superheavy if and only if for any non-positive function $H$ on $M$ with $H\lvert_{K}=0$, we have that $\mu(a; H)=0$.
	\end{enumerate}
\end{lemma}
\begin{proof}
	This is Proposition 4.1 in \cite{EP09}. Note that we have a different sign convention than the original proof. So the positivity and negativity are swapped.
\end{proof}

\begin{remark}
Recall Remark \ref{rem:low-semi-spec}. Using monotonicity and the triangle inequality it is easy to show that for $K$ compact, $c(1;\chi_K)$ is equal to either $0$ or $\infty$. $K$ is heavy precisely when the former possibility is true.

One might be tempted to characterize superheavy sets by $c(1;-\chi_K)=0,$ but this equality in fact only holds for $K=M$. Choosing a non-positive $C^2$-small function that is zero on $K$ and negative somewhere proves this claim. The value of $c(1;-\chi_K)$ could also be a negative real number.
\end{remark}

\subsection{Relative symplectic cohomology}

Now we review the construction of the relative symplectic cohomology, which was introduced in \cite{Var, Var2021}. The construction could be seen as a categorification of the definition of $c(a;\chi_K)$ from Remark \ref{rem:low-semi-spec} using monotone approximation.

Let $M$ be a closed symplectic manifold and $K$ be a compact subset of $M$. Consider $H_{1}\leq H_{2}\leq \cdots$ a monotone sequence of non-degenerate one-periodic Hamiltonian functions $H_{n}: M\times S^{1}\to \RR$ such that
\begin{enumerate}
	\item $H_{n}$ is negative on $K\times S^1$;
	\item for $x\in K$, $H_{n}(x,t)$ converges to $0$ as $n\to \infty$;
	\item for  $x\notin K$, $H_{n}(x,t)$ converges to $+\infty$ as $n\to \infty$.
\end{enumerate}
Then we choose a monotone homotopy of Hamiltonian functions connecting $H_{n}$ and $H_{n+1}$ for each $n$. The sequence of Hamiltonian functions together with chosen monotone homotopies will be called an acceleration datum for $K$.

Now given an acceleration datum for $K$, we denote the differential of $CF(H_{n}; \Lambda)$ by $d_{n}: CF(H_{n}; \Lambda)\to CF(H_{n}; \Lambda)$ and the continuation maps by $h_{n}: CF(H_{n}; \Lambda)\to CF(H_{n+1,}; \Lambda)$.

Both Floer differentials and continuation maps are defined by counting suitable cylinders $u: \RR\times S^{1}\to M$ satisfying the Floer equation. We emphasize that both $d_{n}$ and $h_{n}$ are weighted by powers of the Novikov parameter $T^{E_{top}(u)}$.  Since our homotopies are chosen to be monotone, this topological energy $E_{top}(u)$ is always non-negative for any $u$ in the definitions of $d_{i}$ and $h_{i}$. 

We have a one-ray of Floer complexes
$$
\cC:= CF(H_{1}; \Lambda)\to CF(H_{2}; \Lambda)\to \cdots
$$
and we can form a new $\mathbb{Z}/2$-graded complex $tel(\cC)=\oplus_{n}(CF(H_{n}; \Lambda)\oplus CF(H_{n}; \Lambda)[1])$ by using the telescope construction, see Section 3.7 \cite{AS} and Section 2.1 \cite{Var2021}. The $\Lambda$-linear differential $\delta$ is defined as follows, if $x_{n}\in CF^k(H_{n}; \Lambda)$ then
$$
\delta x_{n}= (-1)^{k}d_{n}x_{n} \in CF^{k+1}(H_{n}; \Lambda),  
$$
and if $x'_{n}\in CF^k(H_{n}; \Lambda)[1]$ then

\begin{align}\label{eq:delta}
	\delta x'_{n}&= ((-1)^{k}x'_{n}, (-1)^{k+1}d_{n}x'_{n}, (-1)^{k+1}h_{n}x'_{n})\\
	&\in CF^{k}(H_{n}; \Lambda)\oplus CF^{k+1}(H_{n}; \Lambda)[1]\oplus CF^k(H_{n+1}; \Lambda).
\end{align}

This telescope comes with a min valution induced by the valuations on $CF(H_{n}; \Lambda)$. Then we have an induced  non-Archimedean norm and an induced metric
\begin{align}\label{eq:norm}
\norm{x}_{\val}:= e^{-\val(x)}, \quad d_{\val}(x, y):= e^{-\val(x-y)}= \norm{x-y}_{\val}.
\end{align}
Finally we complete $tel(\cC)$ using this norm. The resulting complete normed vector space is written as $\widehat{tel}(\cC)$. More concretely, we are adding infinite sums $\sum_{i=1}^{+\infty}(x_{n}+ x'_{n})$ to $tel(\cC)$, where the valuations of $x_{n}$ and $x'_{n}$ go to positive infinity. Since the differential on $tel(\cC)$ is continuous (in fact contracting), it induces an differential on $\widehat{tel}(\cC)$. The relative symplectic cohomology of $K$ in $M$ is defined as the $\delta$-homology
$$
SH_{M}^*(K; \Lambda):= H^*(\widehat{tel}(\cC), \delta).
$$

\begin{remark}
	The above construction also works over the Novikov ring $\Lambda_{\ge 0}$. Let $CF(H_{n}; \Lambda_{\ge 0})$ be the free $\Lambda_{\ge 0}$-module generated by one-periodic orbits of $H_{n}$. By our choice of monotone homotopies, all the Floer differentials and continuation maps are weighted by non-negative powers of $T$. Hence the resulting $T$-adically completed telescope $\widehat{tel}(\cC; \Lambda_{\ge 0})$ is a $\Lambda_{\ge 0}$-module with a differential $\delta$. We call the $\delta$-homology
	$$
	SH_{M}(K; \Lambda_{\ge 0}):= H(\widehat{tel}(\cC; \Lambda_{\ge 0}), \delta)
	$$
	the relative symplectic cohomology over $\Lambda_{\ge 0}$.
	
	Note that the $T$-adic completion of a free $\Lambda_{\ge 0}$-module $A$ tensored with $\Lambda$ can be canonically identified with the completion of   $A\otimes_{\Lambda_{\ge 0}}\Lambda$ with respect to the induced non-Archimedean norm (see e.g.  Equation \eqref{eq-norm}). Since $\Lambda$ is flat over $\Lambda_{\ge 0}$, we have that
	$$
	SH_{M}(K; \Lambda)\cong SH_{M}(K; \Lambda_{\ge 0})\otimes_{\Lambda_{\ge 0}}\Lambda.
	$$
	On the other hand, $SH_{M}(K; \Lambda_{\ge 0})$ does have torsion and carries more quantitative information compared with $SH_{M}(K; \Lambda)$. We will often use the $\Lambda_{\geq 0}$ version in the arguments below.
\end{remark}

In \cite{Var2021}, it is proved that $SH_{M}(K; \Lambda)$ and $SH_{M}(K; \Lambda_{\ge 0})$ are independent of choices. Moreover, they enjoy several good properties, notably a Mayer-Vietoris sequence.

\begin{proposition}[Proposition 3.3.3, \cite{Var2021}]
	Invariance and restriction maps:
	\begin{enumerate}
		\item  For two different acceleration data $H_{s}$ and $H_{s}'$, we have that $H(\widehat{tel}(\cC; \Lambda_{\ge 0}))\cong H(\widehat{tel}(\cC'; \Lambda_{\ge 0}))$ canonically. Therefore we simply denote this invariant by $SH_{M}(K; \Lambda_{\ge 0})$.
		
		\item Let $\phi: M\rightarrow M$ be a symplectomorphism. There exists a canonical isomorphism $SH_{M}(K; \Lambda_{\ge 0})=SH_{M}(\phi(K); \Lambda_{\ge 0})$ by relabeling an acceleration data with the map $\phi$.
		
		\item For $K\subset K'$, there exists canonical restriction maps $r: SH_{M}(K'; \Lambda_{\ge 0})\rightarrow SH_{M}(K; \Lambda_{\ge 0})$. These satisfy the presheaf property.
	\end{enumerate}
\hfill \qedsymbol{}
\end{proposition}

\begin{theorem}
	Global section and the displaceability:
	\begin{enumerate}
		\item (Theorem 1.3.1, \cite{Var}) $SH_{M}^{*}(M; \Lambda_{\ge 0})\cong H^{*}(M; \ZZ)\otimes_{\ZZ} \Lambda_{> 0}$ as $\ZZ/2$-graded $\Lambda_{\ge 0}$-modules.
		\item (Theorem 1.2, \cite{TVar}) If $K$ is stably displaceable then $SH_{M}^{*}(K; \Lambda)= 0$.
	\end{enumerate}\qed
\end{theorem}

\begin{definition}\label{def: Poisson}
	(Definition 1.22, \cite{DGPZ})
	Let $K_{1}, K_{2}$ be two compact sets in $M$. We say they are Poisson commuting if there are two Poisson commuting functions $f_{1}, f_{2}\in C^{\infty}(M; [0,1])$ such that $K_{1}=f^{-1}_{1}(0), K_{2}=f^{-1}_{2}(0)$.
\end{definition}

\begin{theorem}
	Let $K_{1}, K_{2}$ be two compact sets such that they are Poisson commuting. Then there is a long exact sequence
	$$
	\begin{aligned}
		\cdots&\to SH_{M}^{*}(K_{1}\cup K_{2}; \Lambda_{\ge 0})\to SH_{M}^{*}(K_{1}; \Lambda_{\ge 0})\oplus SH_{M}^{*}(K_{2}; \Lambda_{\ge 0})\to SH_{M}^{*}(K_{1}\cap K_{2}; \Lambda_{\ge 0})\to\\
		&\to SH_{M}^{*+1}(K_{1}\cup K_{2}; \Lambda_{\ge 0})\to\cdots.
	\end{aligned}
	$$
\hfill \qedsymbol{}
\end{theorem}

We refer to Theorem 1.3.4 \cite{Var2021} for the most general form of this theorem.

\begin{theorem}[\cite{TVar}]
Relative symplectic cohomology $SH_{M}(K; \Lambda)$ has a canonical product structure and unit $e_{K}$ which makes it a graded commutative unital $\Lambda$-algebra. Restriction maps respect the unit and the product structure.
\hfill \qedsymbol{}
\end{theorem}

%We also have canonical PSS-type maps $$PSS_K: QH(M;\Lambda)\to SH_{M}(K; \Lambda),$$ which respect the product structure. These are defined using the homotopy commutative diagram (with prescribed homotopies, i.e. a map of $1$-rays):
%
%$$
%	\begin{tikzcd}
%	\ldots\ar[r]&CM(f, \Lambda)\ar[r]\ar[d]\ar[dr]&CM(f, \Lambda)\ar[r]\ar[d]\ar[dr]&CM(f, \Lambda)\ar[r]\ar[d]&\ldots\\
%	\ldots\ar[r]&CF(H_{n-1}; \Lambda)\ar[r]&CF(H_{n}; \Lambda)\ar[r]&CF(H_{n+1}; \Lambda)\ar[r]&\ldots
%	\end{tikzcd}
%	$$
%with the vertical maps the chain level PSS maps as above and the upper horizontal maps being the identity. 

The techniques of \cite{TVar} can also be used to show the following result.

\begin{theorem}
There are canonical unital $\Lambda$-algebra maps $$PSS_K: QH(M;\Lambda)\to SH_{M}(K; \Lambda),$$
given by the composition $QH(M;\Lambda)\to SH_{M}(M; \Lambda)\to SH_{M}(K; \Lambda)$, which are compatible with the restriction maps.
\hfill \qedsymbol{}
\end{theorem}
%
%It of course follows that $PSS_K(1)=e_K$. 
The following is also standard.

\begin{theorem}
For every compact $K\subset M$ and $H:M\times S^1\to\mathbb{R} $ a smooth function which is negative on $K$, there are canonical $\Lambda_{\geq 0}$-module maps $$\kappa: HF(H; \Lambda_{\geq 0})\to SH_M^*(K;\Lambda_{\geq 0}).$$ These maps satisfy $$\kappa_\Lambda\circ PSS^H_\Lambda=PSS_K,$$ where $\kappa_\Lambda=\kappa\otimes_{\Lambda_{\geq 0}}\Lambda.$
\hfill \qedsymbol{}
\end{theorem}

\begin{definition}\label{def: a-visible}
	Let $K$ be a compact set in $M$. Letting $a\in QH(M; \Lambda)$ denote a non-zero element, we make the following definition.
	\begin{enumerate}
		\item $K$ is called SH-$a$-visible if $PSS_K(a)\neq 0$, otherwise it is called SH-$a$-invisible.
		\item $K$ is called SH-$a$-full if every compact set contained in $M-K$ is SH-$a$-invisible.
		\item $K$ is called nearly SH-$a$-visible if any compact domain containing $K$ in its interior is SH-$a$-visible.
	\end{enumerate}
	
When $a$ is the unit, we omit it from the notation.
\end{definition}

By unitality, SH-$a$-visiblity is of course equivalent to $PSS_K(a)\cdot SH_{M}(K; \Lambda)\neq 0,$ and in particular SH-visibility is equivalent to $SH_{M}(K; \Lambda)\neq 0.$ When $a$ is an idempotent, so is $PSS_K(a)$. Therefore, $PSS_K(a)\cdot SH_{M}(K; \Lambda)\subset SH_{M}(K; \Lambda)$ is an ideal, which is also on its own an algebra with unit $PSS_K(a)$. Our results in this paper will only concern the case that $a$ is an idempotent.

Let us finally note that for an idempotent $A\in SH_{M}(K; \Lambda)$, we have the following finer Mayer-Vietoris sequence (under the same assumptions):

$$
	\begin{aligned}
		\cdots&\to A\cdot SH_{M}^{*}(K_{1}\cup K_{2}; \Lambda)\to A\cdot SH_{M}^{*}(K_{1}; \Lambda)\oplus A\cdot SH_{M}^{*}(K_{2}; \Lambda)\to A\cdot SH_{M}^{*}(K_{1}\cap K_{2}; \Lambda)\to\\
		&\to A\cdot SH_{M}^{*+1}(K_{1}\cup K_{2}; \Lambda)\to\cdots.
	\end{aligned}
	$$ See Remark 1.8 of \cite{TVar} for a detailed discussion.

\subsection{Symplectic ideal-valued quasi-measures}
Let $K$ be a compact set of a closed symplectic manifold $M$. Dickstein-Ganor-Polterovich-Zapolsky \cite{DGPZ} defined the quantum cohomology ideal-valued quasi-measure of $K$ as
\begin{equation}
	\tau(K):= \bigcap_{K\subset U} \ker(r: SH_{M}(M; \Lambda)\rightarrow SH_{M}(M-U; \Lambda))
\end{equation}
where $U$ runs over all open sets containing $K$.

\begin{definition}
	 If $\tau(K)\neq 0$, then we say $K$ is SH-heavy.
\end{definition}

%Our results show that SH-visibility of $K$ is equivalent to heaviness. Therefore we will focus on that notion and only make remarks about the relationship to SH-heaviness - leaving further discussion to other works.

\subsection{Chain level product structure}

Let $\Bbbk$ be a commutative ring.
Let us call a $\mathbb{Z}/2$-graded $\Bbbk$-algebra, not necessarily associative or commutative, equipped with a differential satisfying the graded Leibniz rule a \textit{$\Bbbk$-chain complex with product structure}.

Using the results of \cite{AGV}, the following theorem is immediate.

\begin{theorem}\label{thm-AGV}
For every compact $K\subset M$, we can construct a $\Lambda_{\geq 0}$-chain complex $\mathcal{C}_K$ with product structure whose underlying $\Lambda_{\geq 0}$-module is torsion free and $T$-adically complete such that 

\begin{enumerate}
\item There are preferred isomorphisms of $\mathbb{Z}/2$-graded $\Lambda_{\geq 0}$-algebras $$\iota_K:SH_M^*(K;\Lambda_{\geq 0})\to H^*(\mathcal{C}_K).$$
\item Let $H:M\times S^1\to\mathbb{R} $ be a non-degenerate smooth function which is negative on $K$. Then there is a $\Lambda_{\geq 0}$-chain map $$CF(H; \Lambda_{\geq 0})\to \mathcal{C}_K$$ such that $$\begin{tikzcd}[column sep=small]
 \arrow{r}{\kappa}  \arrow{rd}
 HF(H; \Lambda_{\geq 0}) & SH_M^*(K;\Lambda_{\geq 0}) \arrow{d}{\iota_K} \\
    & H^*(\mathcal{C}_K )
\end{tikzcd}$$ commutes.
\end{enumerate} 
\end{theorem}

\begin{proof}
We define $$\mathcal{C}_K^*:=SC^*_{M,f\Mbar^{\RR}_0}(K)$$ as in Theorem 1.4 of \cite{AGV}. The product comes from an arbitrary closed element of $C_0(f\Mbar^{\RR}_{0,2+1};\Lambda_{\geq 0})$ which generates $H_0(f\Mbar^{\RR}_{0,2+1};\Lambda_{\geq 0})\simeq \Lambda_{\geq 0}$, see the discussion surrounding Equation (1.3) in \cite{AGV}. Then Theorem 1.8 and 1.9 therein prove the remaining statements.
\end{proof}

\begin{remark}
Actually, we do not even need $\iota_K$ to respect the product structures for our application. All we need is that the composition $$QH(M; \Lambda)\to HF(H; \Lambda) \to H^*(\mathcal{C}_K\otimes_{\Lambda_{\geq 0}}\Lambda )$$ respects the product structures.
Theorem \ref{thm-AGV} implies this composition respects the product structures because it equals to
$$QH(M; \Lambda)\to SH_M^*(K;\Lambda) \to H^*(\mathcal{C}_K\otimes_{\Lambda_{\geq 0}}\Lambda )$$
which is a composition of algebra maps.
\end{remark}

\begin{remark}
It should be possible to construct a product directly on the telescope model over $\Lambda_{\geq 0}$ (and hence the completed telescope). We know this a posteriori because using homotopy transfer results as in Corollary 1.13 of \cite{AGV} one can equip it with an $A_\infty$-structure (and under a characteristic $0$ assumption with a $BV_\infty$-structure). We stress that this also involves choosing extra data: for example, a one sided homotopy inverse to the quasi-isomorphism from the telescope to the big model. The existence of this data is proven using model theoretic arguments in \cite{AGV} without an explicit construction.

To show the difficulty let us try to multiply an element from $CF(H_n; \Lambda_{\geq 0})$ with an element from $CF(H_m; \Lambda_{\geq 0})$ in the telescope (the degrees unshifted copies). The product could plausibly live in any $CF(H_k; \Lambda_{\geq 0})$ with $H_k$ sufficiently larger than $H_n$ and $H_m$ (see Definition 2.18 of \cite{AGV} for details). Yet, which one? Or should it have more than one component perhaps? We do not know how to proceed with a direct approach like this. We end with a question: is there a product on the telescope model over $\Lambda_{\geq 0}$ which satisfies the Leibniz rule and induces the correct homology level product; and moreover, only involves moduli problems with at most two inputs?
\end{remark}

\section{Proof of results}

\subsection{An algebraic lemma}
We first state an algebraic lemma, which will play a key role in the proof of our main theorem. The construction of an algebraic structure on a suitable chain model of the relative symplectic cohomology satisfying this lemma is in \cite{AGV}, and what we need is Theorem \ref{thm-AGV}. 
%It will be summarized somewhere.

Let $(A, d)$ be a $\Lambda$-chain complex with a product structure. This implies the existence of a graded product on $H(A)$, the cohomology of $A$. Assume that there is a map $\val: A\to \RR\cup \lbrace +\infty\rbrace$ such that $\val^{-1}(\infty)=0$ \footnote{We warn the reader that for the induced valuation on $H^*(A)$ only the weaker $\val(0)=\infty$ is necessarily true.} and for $c\in \Lambda$ and $x,y\in A$:
\begin{enumerate}
	\item $\val (x\cdot y)\geq \val(x) +\val(y)$;
	\item $\val (dx)\geq \val(x)$;
	\item $\val (cx)= \val(c)+\val(x)$;
	\item $\val (x+y)\geq \min\lbrace \val(x), \val(y)\rbrace$.	
\end{enumerate}
Moreover we assume that $A$ is complete with respect to the non-archimedean norm induced by $\val$. 

A $\Lambda_{\geq 0}$-chain complex $\mathcal{C}$ with product structure whose underlying $\Lambda_{\geq 0}$-module is $T$-adically complete gives such a structure on $\mathcal{C}\otimes_{\Lambda_{\geq 0}}\Lambda. $ Here the valuation on $\mathcal{C}\otimes_{\Lambda_{\geq 0}}\Lambda$ is defined by \begin{equation}\label{eq-norm} val(x):=-\inf\{r\in \mathbb{R}\mid T^rx\in \iota(C)\}, \end{equation}where $\iota: C\to C\otimes \Lambda$ is the natural map. 

We define $xyz:= x\cdot(y\cdot z)$. And $x^{k}$ should be understood similarly for any positive integer $k$.

\begin{lemma}\label{l: algebaric lemma}
	Let $(A, d, \val)$ be as above. Let $a$ be a non-zero idempotent in $H(A)$. Then any degree $0$ chain level representative of $a$ has non-positive valuation.
\end{lemma}
\begin{proof}
	Pick a closed homogeneous element $y$ with $[y]=a$ and $\deg(y)=0$. Then we have that $[y\cdot y]= [y]\cdot [y] =[y]$. So there exist $\alpha\in A$ such that
	$$
	y-y\cdot y= d\alpha.
	$$
	Multiply $y$, from the left, to both sides of the equation we get
	$$
	y\cdot y- y\cdot (y\cdot y)= y\cdot d\alpha= d(y\cdot \alpha).
	$$ 
	Multiply $y$, from the left, to this equation we get
	$$
	y\cdot (y\cdot y)- y\cdot (y\cdot (y\cdot y))= y\cdot (y\cdot d\alpha)= d(y\cdot (y\cdot \alpha)).
	$$
	Then we can repeat this process to get
	$$
	y^{k}- y^{k+1}= d(y^{k-1}\alpha), \quad \forall k\geq 1.
	$$
	Formally summing over $k$ we get
	$$
	y=\sum_{k=0}^{+\infty} (y^{k+1}-y^{k+2})= d\sum_{k=0}^{+\infty} y^{k}\alpha.
	$$
	If $\val(y)>0$, then the term $y^{k}\alpha$ has valuation going to infinity, as $k$ goes to infinity. Hence $\sum_{k=0}^{+\infty} y^{k}\alpha$ is convergent. This implies that $[y]=0$, a contradiction. 
%	
%	Generally,\footnote{do we need this case?} if $y$ is closed with $[y]=a$ but not homogeneous in degree, we can break it into a sum of degree-zero part and the other part: $y=y_{0}+ y'$. Then $y_{0}$ is also a chain level representative for $a$. Since there is no degree-zero part in $y'$, we have $\val(y)= \min\lbrace \val(y_{0}), \val(y')\rbrace$. Hence $\val(y)>0$ implies $\val(y_{0})>0$. And we can use the above process for $y_{0}$ to get a contradiction.
\end{proof}

In fact, we can prove the more general result below by the same proof.

\begin{lemma}\label{l: finer algebaric lemma}
	Let $(A, d, \val)$ be as above. Let $a, b\in H(A)$ satisfy $ba=\lambda a$ with $\lambda\in \Lambda^*$ and assume that $b$ has a representative with valuation larger than $val(\lambda).$ Then $a=0$.
\hfill \qedsymbol{}
\end{lemma}

To recover the previous lemma, set $a=b$ and $\lambda=1$. The reader is invited to compare our discussion with \cite[Section 1.6.2]{BSV}.

%\begin{remark}
%	The above lemma does not use the associativity of the algebra $A$. This is important when we apply it to the relative symplectic cohomology, since the associativity of a chain level product does not generally hold.
%\end{remark}

\subsection{Heaviness}
Now we use this algebraic lemma to prove one of our main theorems. Let $a\in QH(M; \Lambda)$ be a non-zero idempotent throughout this section.

\begin{theorem}\label{t: heavy1}
For a compact subset $K$ of $M$, let $H:M\to \mathbb{R}$ be a smooth function which is zero on $K$.  If $c(a; H)>0$ then $PSS_{K}(a)=0\in SH_{M}(K; \Lambda)$.
\end{theorem}
\begin{proof}
We can easily find a non-degenerate smooth function $H':M\times S^1\to \mathbb{R}$ that is negative on $K$ with $c(a; H')>0$. We consider the map $g: CF(H'; \Lambda)\to \mathcal{C}_K\otimes_{\Lambda_{\geq 0}}\Lambda$ as in Theorem \ref{thm-AGV}, which gives us the following commutative diagram.

$$
\begin{tikzcd}[column sep= large]
	QH(M; \Lambda) \arrow{r}{PSS_{K}} \arrow{d}{PSS^{H'}_{\Lambda}} & SH_M^*(K;\Lambda) \arrow{d}{\iota_K} \\
	HF(H'; \Lambda) \arrow{r}{[g]} \arrow{ur}{\kappa} & H^*(\mathcal{C}_K\otimes_{\Lambda_{\geq 0}}\Lambda)
\end{tikzcd}
$$

Let $A$ be a representative of $PSS^{H'}_\Lambda(a)$ in $CF(H'; \Lambda)$ with positive valution, which exists because $c(a; H')>0$. By Theorem \ref{thm-AGV} it suffices to show that $[g(A)]=0.$ Since $g$ respects the product structures at the homology level, $  [g(A)]$ is an idempotent. For the natural valuation on $\mathcal{C}_K\otimes_{\Lambda_{\geq 0}}\Lambda,$ we have $val(g(A))\geq val(A)>0$. We now use Lemma \ref{l: algebaric lemma} to finish.
\end{proof}

 In particular, if $c(1; H)>0$ then $SH_{M}(K; \Lambda)=0$. A direct corollary is that an SH-visible compact set is heavy.

\begin{corollary}\label{co: heavy1}
For a compact subset $K$ of $M$, if $K$ is SH-$a$-visible then $K$ is $a$-heavy.
\end{corollary}
\begin{proof}
	Suppose that $K$ is not $a$-heavy, then there exists a non-negative function $H$ which is zero on $K$, with $\mu(a; H)> 0$ by Lemma \ref{l: criteria}. However, Theorem \ref{t: heavy1} implies that $c(a; nH)\leq 0$ (in fact this is an equality) for all $n\in \mathbb{Z}$, which is a contradiction.
\end{proof}

%By Lemma \ref{l:opennbhd}, we get a slightly stronger criterion for heaviness.
%
%\begin{corollary}\label{co: heavy2}
%For a compact subset $K$ of $M$, if $K$ is nearly SH-visible, then $K$ is heavy.
%\end{corollary}
%\begin{proof}
%	For any compact domain $\bar{U}$ containing $K$ in its interior, the above corollary says that $\bar{U}$ is heavy. Then the approximation Lemma \ref{l:opennbhd} shows that $K$ is heavy.
%\end{proof}

Next we prove that heavy implies SH-visible.

\begin{theorem}\label{t: heavy2}
	Let $K$ be a compact subset of $M$. If $K$ is $a$-heavy, then $PSS_K(a)\neq 0$.
\end{theorem}
\begin{proof}
	Let $G$ be a non-negative function on $M$ that vanishes precisely on $K$. Since $K$ is $a$-heavy, using Lemma \ref{l:equivheavy} and the monotonicity property of spectral invariants, we have that $c(a; nG)=0$.

Note that $G$ is strictly positive on $M-K$, so $nG$ goes to positive infinity on $M-K$ as $n$ goes to infinity. Hence, we can find acceleration data $\lbrace G_{n}\rbrace$ for $K$, such that
	$$
	G_{n}(x,t)< nG(x), \quad \forall n\geq 1, \quad \forall x\in M.
	$$
	By the monotonicity property of spectral invariants, we have that $0>c(a; G_{n})$ for any $n\geq 1$. 
	
	Let $x\in CF(G_{1})$ be an element representing $PSS^{G_1}_{\Lambda}(a)$. It is easy to show that 
	$$
	(x, 0, 0, 0, \cdots)\in \widehat{tel}(CF(G_{n};\Lambda))
	$$
	is a closed element that represents $PSS_K(a)$ in $SH_{M}(K; \Lambda)$. We 
	recall for convenience from \eqref{eq:delta} that the differential of the telescope is defined as
	\begin{align}
	\delta(x_{1}, x'_{1}, x_{2}, x'_{2}, \cdots)= (d_{1}x_{1}+ x'_{1}, -d_{1}x'_{1}, d_{2}x_{2}+ x'_{2}+ h_{1}x'_{1}, -d_{2}x'_{2}, \cdots)  \label{eq:TelDiff}
	\end{align}
	where $d_{n}$ is the Floer differential on $CF(G_{n}; \Lambda)$ and $h_{n}: CF(G_{n}; \Lambda)\to CF(G_{n+1}; \Lambda)$ is the continuation map.
	
	Suppose that
	$$
	\delta(y_{1}, y'_{1}, y_{2}, y'_{2}, \cdots)= (x, 0, 0, 0, \cdots),
	$$
	which is equivalent to
	\begin{align}
	d_{1}y_{1}+ y'_{1}=x, \quad d_{1}y'_{1}=0, \quad d_{2}y_{2}+ y'_{2}+ h_{1}y'_{1}=0, \quad d_{2}y'_{2}=0, \cdots. \label{eq:TelRel}
	\end{align}
	The first equation shows that $y'_{1}$ is homologous to $x_{1}$ in $HF(G_{1}; \Lambda)$, and hence $\val(y'_{1})<0$. The third equation shows that $y'_{2}$ is homologous to $h_{1}y'_{1}$ in $HF(G_{2}; \Lambda)$. By the compatibility of continuation and PSS maps, $h_{1}y'_{1}$ represents $PSS^{G_2}_{\Lambda}(a),$ and hence $\val(y'_{2})<0$. By induction, using the same argument, we get that $\val(y'_{n})<0$ for all $n\geq 1$. Therefore $(y_{1}, y'_{1}, y_{2}, y'_{2}, \cdots)$ is not an element of the completed telescope, which shows that $(x, 0, 0, 0, \cdots)$ is indeed not exact.
\end{proof}

\begin{remark}\label{re: infinity}
	The proof above actually shows that if $(y_{1}, y'_{1}, y_{2}, y'_{2}, \cdots)$ is a convergent primitive of $(x_{1}, 0, 0, 0, \cdots)$, then $\lim_{n\to +\infty}c(a; G_{n})= +\infty$. Because $\val(y_{n}')\to +\infty$ as $n\to +\infty$ and $y'_{n}$ represents $PSS^{G_{n}}_{\Lambda}(a)$. Here $a$ can be any non-zero class in $QH(M; \Lambda)$, not necessarily an idempotent. This fact may be of independent interest, and will be used later.
\end{remark}

Now Corollary \ref{co: heavy1} and Theorem \ref{t: heavy2} give the equivalence between heavy sets and SH-visible sets, in full generality. Next we discuss some immediate consequences.

\begin{lemma}\label{l:opennbhd}
	Let $K \subset M$ be a compact set and consider a sequence of compact sets $K_1\supset K_2\supset ...$ such that $\bigcap K_n=K.$
	If $K_n$ is $a$-heavy (resp. $a$-superheavy) for all $n\geq 1$, then $K$ is $a$-heavy (resp. $a$-superheavy).

\end{lemma}
\begin{proof}
%	Pick a nested sequence of open sets $U_1 \subset U_2 \subset \dots$ such that $\cup_k U_k =M \setminus K$, $\bar{U_k} \subset U_{k+1}$ and $\overline{M\setminus \bar{U_k}}=M \setminus U_k$.
%	By assumption, $\mu(1,H) \le \max_{M \setminus U_k} H$ (resp. $\ge \min_{M \setminus U_k} H$) for all $k$.
%	Since $H$ is continuous, it implies that $\mu(1,H) \le \max_{K} H$ (resp. $\ge \min_{K} H$) so the result follows. 
%	
We only prove $a$-heaviness. The proof of $a$-superheaviness is similar.

We will use Lemma \ref{l: criteria} again to certify that $K$ is heavy. Let $H$ be a non-negative smooth function on $M$ which vanishes on $K$. Define $\epsilon_{n}$ as the maximum of $H$ on $K_n$. Clearly we have $\epsilon_n\to 0$ as $n\to\infty$. Since $K_{n}$ is $a$-heavy for each $n$, then $\mu(a;H)\leq \max_{K_{n}}H= \epsilon_{n}$. So $\mu(a;H)=0$ and $K$ is $a$-heavy. 
\end{proof}

On the other hand, it is not clear that how the relative symplectic cohomology would behave under approximations. Particularly, whether the two notions SH-visible and nearly SH-visible are equivalent was not known previously. Now we have an affirmative answer by combining Corollary \ref{co: heavy1}, Theorem 
\ref{t: heavy2} and Lemma \ref{l:opennbhd}.

\begin{corollary}\label{co: nearly visible}
	A compact subset $K$ of $M$ is SH-$a$-visible if and only if it is nearly SH-$a$-visible. 
\end{corollary}
\begin{proof}
Note that for any compact $K$ we can construct compact domains $K_1\supset K_2\supset ...$ such that $\bigcap K_n=K$, for example by finding a non-negative smooth function on $M$ whose vanishing set is precisely $K$ and using Sard's theorem. If $K$ is nearly $SH$-visible, then $K_n$ is $SH$-visible and hence heavy for all $n$. Lemma \ref{l:opennbhd} shows that $K$ is heavy and therefore $SH$-visible.
\end{proof}

Another application is about the notion of SH-heaviness.

\begin{proposition}\label{p:basicheavy}
	If a compact subset $K$ of $M$ is SH-heavy, then it is heavy.
\end{proposition}
\begin{proof}
	For a compact set $K$ and an open set $U$ containing $K$, we can find another open set $V$ such that $M-U\subset V$, $\bar{V}\cap K=\emptyset$ and $\partial \bar{V}\cap \partial \bar{U}=\emptyset$. If $SH_{M}(\bar{U}; \Lambda)= 0$, then the Mayer-Vietoris sequence for the pair $(\bar{U}, \bar{V})$ tells us that 
	$$
	\ker(r: SH_{M}(M; \Lambda) \to SH_{M}(\bar{V}; \Lambda))= 0.
	$$
	Hence for an SH-heavy set $K$, any open set $U$ containing $K$ has $SH_{M}(\bar{U}; \Lambda) \neq 0$. We deduce that $K$ is nearly SH-visible.  Then Corollaries \ref{co: nearly visible} and \ref{co: heavy1} imply that $K$ is heavy.
\end{proof}

This answers one direction of Conjecture \ref{con: DGPZ} in full generality. We do not know the answer for the opposite direction.

%commutative diagram
%\begin{equation}\label{eq:Fmap}
%\xymatrix{
%SH_M(M,0) \ar[r] \ar[d]  & SH_M(\bar{U},0)  \ar[d]\\
%SH_M(M,H) \ar[r]  &  SH_M(\bar{U}, \max_{\bar{U}} H)
%}
%\end{equation}
%Let the image of $1$ in $SH_M(\bar{U},0)$, $SH_M(M,H)$ and  $SH_M(\bar{U}, \max_{\bar{U}} H)$ be $f$, $f'$ and $f''$ respectively.
%Earlier, we have shown $c(f)=0$ and the same proof would imply $c(f'')=\max_{\bar{U}} H$.
%On the other hand, we have $c(f') \le c(f'')$ so we have $c(f') \le \max_{\bar{U}} H$. It gives the heaviness of $\bar{U}$.
%By considering all choices of $U$, we have $c(f') \le \max_{K} H$. It gives the heaviness of $K$.

\subsection{Superheaviness}\label{ss-superheavy}
Our first result is a corollary of the fact that a superheavy set intersects any heavy set. 
% throughout this section.\footnote{a missing sentence?}

\begin{corollary}\label{co: superheavy1}
	Let $a\in QH(M; \Lambda)$ be a non-zero idempotent. For a compact subset $K$ of $M$, if $K$ is $a$-superheavy then $K$ is SH-$a$-full.
\end{corollary}
\begin{proof}
	Suppose that $K$ is not SH-$a$-full. Then there exists some compact subset $K'$ of $M$ which is disjoint from $K$ and SH-$a$-visible. By Corollary \ref{co: heavy1}, $K'$ is $a$-heavy, which contradicts to that $K$ is $a$-superheavy as $a$-heavy sets intersect $a$-superheavy sets.
\end{proof}

We do not know how to prove the converse in general, but we will discuss a strategy that gives a partial converse. We restrict to $a=1$ for simplicity in the rest of the paper. We start with the key but simple lemma.

%\begin{lemma}
%Assume that we have
%	$$
%	\sup \lbrace c(a;H)\mid \Pi(a,1)\neq 0\rbrace =c([vol]; H)
%	$$
%	for some smooth function $H$ (*). Then, $$c(1; H)=-c([vol]; -H).$$
%\end{lemma}

\begin{lemma}\label{l: cy}
	Let $(M, \omega)$ be a symplectic manifold with $c_1(M,\omega)|_{\pi_2(M)}=0$. We have that
	\begin{align}\label{eq:CYPoincare}
	c([vol]; H)=\sup \lbrace c(a;H)\mid \Pi(a,1)\neq 0\rbrace
	\end{align}
	for any $H \in C^{\infty}(M)$. As a consequence, we have 
$$c(1; H)=-c([vol]; -H).$$
\end{lemma}
\begin{proof}
	First we have that $\Pi([vol],1)\neq 0$. Hence $\sup \lbrace c(a;H)\mid \Pi(a,1)\neq 0\rbrace \geq c([vol]; H)$.
	
	Next we prove an inequality in the other direction. Pick a class $a$ with $\Pi(a,1)\neq 0$. Since $(M, \omega)$ is symplectic Calabi-Yau on spherical classes, the quantum product respects the natural $\mathbb{Z}$-grading on $QH(M,\Lambda)$. 
%the Novikov variable $T$ has degree zero 
In particular, we need to have
	$$
	a= a_{0}[vol]+ a'
	$$
	where the degree of $a'$ is less than $2n$, and $a_{0}\in \Lambda$ with $\val(a_{0})= 0$. Recall for two classes $x,y$ with different degrees, the spectral invariants satisfy that  
	$$
	c(x+y; H)= \min\lbrace c(x; H), c(y; H)\rbrace
	$$
	for any function $H$. So we get
	$$
	c(a; H)= \min\lbrace c(a_{0}[vol]; H), c(a'; H)\rbrace \leq c(a_{0}[vol]; H)= c([vol]; H).
	$$
	This shows that $\sup \lbrace c(a;H)\mid \Pi(a,1)\neq 0\rbrace \leq c([vol]; H)$.
The last implication follows from Lemma \ref{l:Poincare}.
\end{proof}

\begin{remark}The equation (\ref{eq:CYPoincare}) also holds when $M$ is
	\begin{enumerate}
		
		\item negatively monotone, or
		\item positively monotone and its minimal Chern number is greater than $n$, where $2n$ is the real dimension of $M$.
	\end{enumerate}
\end{remark} 

Recall that in order to prove that $K$ is superheavy, we need to estimate $c(1;nH)$ as $n\to \infty$ for all non-positive $H$ that is zero on $K$. Using the lemma (when we can), we can instead try to estimate $c([vol];nH)$ as $n\to \infty$ for all non-negative $H$ that is zero on $K$. This is the content of Proposition \ref{p:PSS=0}.
Before we explain it, we recall that the boundary depth (cf. \cite[Section 2]{U2}) of a cochain complex $A$ over $\Lambda$ is defined to be
\begin{align}\label{eq:Bdepth}
\beta(A):= \sup_{x \in \im(\delta)} \inf_{y \in A} \left\{\val(x)-\val(y) | dy=x \right\}.
\end{align}

%\begin{remark}
%The boundary depth in \cite[Section 2]{U2} is defined for action-filtered complexes, for example, $CF(H)$.
%The argument in Proposition \ref{p:compareSpectral} shows that for any $x \in CF(H)$ and its image $\tilde{x} \in CF(H) \otimes_{\Lambda_{\omega}} \Lambda=CF(H;\Lambda)$, we have
%\[
%\inf_{y \in CF(H)} \left\{A(x)-A(y) | dy=x \right\}=\inf_{\tilde{y} \in CF(H;\Lambda)} \left\{\val(\tilde{x})-\val(\tilde{y}) | d\tilde{y}=\tilde{x} \right\}.
%\]
%It implies that $\beta(CF(H)) \le \beta(CF(H;\Lambda))$. Conversely, any element $\tilde{x} \in CF(H;\Lambda) \cap \im(\delta)$ can be written as a linear combination $\tilde{x}=\sum_{q \in \mathbb{N}} \tilde{x}_q$ such that $\tilde{x}_q \in T^{a_q} \im(CF(H) \to CF(H;\Lambda))$, $\tilde{x}_q \in \im(\delta)$ and $\val(\tilde{x}_{q+1}) \ge \val(\tilde{x}_{q}) \ge \val(\tilde{x})$ for all $q$.
%It implies that  $\beta(CF(H))= \beta(CF(H;\Lambda))$.
%\end{remark}

\begin{proposition}\label{p:PSS=0}
Assume that for $a\in QH(M; \Lambda)$ and a non-negative $H \in C^{\infty}(M)$ that is zero on $K$, we have \begin{itemize}
\item $c(a;nH)\to \infty$ as $n\to \infty$, and
\item there is a uniform upper bound on the boundary depths of $CF(nH;\Lambda)$ that is independent of $n$.
\end{itemize}
Then, $PSS_K(a)=0.$

\end{proposition}

\begin{proof}
This proof shares some similarities with the proof of Theorem \ref{t: heavy2}. Instead of showing that a chain level representative of $PSS_K(a)$ is not exact, we want to find a primitive of it using the conditions on $c(a;nH)$ and the boundary depth.

Let ${H_n}$ be an acceleration data of $K$ such that 
\begin{itemize}
\item $H_{n}(x,t)< nH(x), \quad \forall n\geq 1, \quad \forall x\in M$,
\item $c(a;H_n)\to \infty$ as $n\to \infty$, and
\item there is $B>0$ such that the boundary depth $\beta(H_n):=\beta(CF(H_n;\Lambda)) \le B$ for all $n$. 
\end{itemize}

	Let $x\in CF(H_{1})$ be an element representing $PSS^{G_1}_{\Lambda}(a)$, so 
	$$
	(x, 0, 0, 0, \cdots)\in \widehat{tel}(CF(H_{n};\Lambda))
	$$
	is a closed element that represents $PSS_K(a)$ in $SH_{M}(K; \Lambda)$. The differential of the telescope is given by
\eqref{eq:TelDiff}:
	$$
	\delta(x_{1}, x'_{1}, x_{2}, x'_{2}, \cdots)= (d_{1}x_{1}+ x'_{1}, -d_{1}x'_{1}, d_{2}x_{2}+ x'_{2}+ h_{1}x'_{1}, -d_{2}x'_{2}, \cdots)
	$$
	
	We want to construct an element $(y_{1}, y'_{1}, y_{2}, y'_{2}, \cdots)$ such that
	$$
	\delta(y_{1}, y'_{1}, y_{2}, y'_{2}, \cdots)= (x, 0, 0, 0, \cdots),
	$$
	(see\eqref{eq:TelRel} for the equivalent form).
The $y_i$ and $y_i'$ will be defined inductively.

	The first equation of \eqref{eq:TelRel} requires that $y'_{1}$ is homologous to $x_{1}$ in $HF(H_{1}; \Lambda)$.
In particular, we can take $y_1=0$ and $y'_1=x$.
 The third equation of \eqref{eq:TelRel}  shows that $y'_{2}$ is homologous to $h_{1}y'_{1}$ in $HF(H_{2}; \Lambda)$. By the compatibility of continuation and PSS maps, $h_{1}y'_{1}$ represents $PSS^{H_2}_{\Lambda}(a),$ and hence 
we can choose $y_2'$ to be a chain level representative of $PSS^{H_2}_{\Lambda}(a)$ such that $\val(y_2')=c(a;H_2)$. 
Since $y_2'$ and $h_{1}y'_{1}$ descend to the same class in $HF(H_{2}; \Lambda)$, we can find $y_2$ which solves the equation $d_2y_2+y_2'+h_1y_1'=0$.
Moreover, by the definition of the boundary depth \eqref{eq:Bdepth}, for any $\epsilon>0$, we can find such $y_2$ so that $\val(y_2'+h_1y_1')-\val(y_2)<\beta(H_2)+ \epsilon$.
 By induction, using the same argument, we get that for all $n\geq 1$,
\begin{itemize}
\item $\val(y'_{n})=c(a;H_n)$  and
\item $\val(y_n'+h_{n-1}y_{n-1}')-\val(y_n)<\beta(H_n)+ \epsilon \le B+\epsilon$ .
\end{itemize}
Since $\val(h_{n-1}y_{n-1}') \ge \val(y_{n-1}')$ and $\val(y_n'+h_{n-1}y_{n-1}') \ge \min\{y_n',h_{n-1}y_{n-1}'\}$, the fact that $c(a;H_n)$ goes to infinity as $n$ goes to infinity implies that $\val(y_n)$ goes to infinity as $n$ goes to infinity.
Therefore $(y_{1}, y'_{1}, y_{2}, y'_{2}, \cdots)$ is an element of the completed telescope, which shows that $(x, 0, 0, 0, \cdots)$ is exact.

\end{proof}

\begin{theorem}\label{t:SHsuperheavy}
Let $K$ be a compact subset and $H \in C^{\infty}(M)$ is a non-negative function that is zero on $K$.
Suppose that $nH$ satisfies \eqref{eq:CYPoincare} for all $n \in \mathbb{N}$ and there is a uniform upper bound on the boundary depth of $CF(nH;\Lambda)$.
If $PSS_K([vol])\neq 0$, then $K$ is superheavy.
\end{theorem}

\begin{proof}
By $PSS_K([vol])\neq 0$ and Proposition \ref{p:PSS=0},  we know that $c([vol];nH)$ is uniformly bounded above.
Under condition \eqref{eq:CYPoincare} for all $nH$, we therefore obtain that $c(1;-nH)$ is uniformly bounded below and hence $\mu(1;-H)=0.$ It therefore implies the superheaviness of $K$.
\end{proof}

\begin{remark}
One way to achieve a uniform bound on boundary depth is to impose an index bounded condition when $K$ is a Liouville domain.
See \cite{M2020,TVar,DGPZ,S} for more discussions.
\end{remark}

\begin{remark}
Without the condition \eqref{eq:CYPoincare}, we would need a uniform bound on all $c(a;nH)$ with $\Pi(a,1)\neq 0$. This does not follow from $PSS_K(a)\neq 0$ for all such $a$ and appears more difficult to obtain. On the other hand what we ended up proving is strictly stronger than $\mu(1;-H)=0.$
\end{remark}

There is a convenient way to deal with the assumption on boundary depth by replacing relative symplectic cohomology with reduced symplectic cohomology \cite{Gro}, which is defined as follows:
take $\{H_n\}$ to be an acceleration data of $K$ and form the completed telescope $\widehat{tel}(CF(H_{n};\Lambda))$, then the reduced symplectic cohomology 
$SH_M^{red,*}(K;\Lambda)$
is defined to be the kernel $\ker(\delta:\widehat{tel}^{*}(CF(H_{n};\Lambda)) \to \widehat{tel}^{*+1}(CF(H_{n};\Lambda)))$ modulo the closure of the image $\overline{\im}(\delta:\widehat{tel}^{*-1}(CF(H_{n};\Lambda)) \to \widehat{tel}^{*}(CF(H_{n};\Lambda)))$
where $\overline{\im}$ refers to the closure of $\im$ with respect to \eqref{eq:norm}. 

%Tal is used in $L^2$-cohomology theory \cite{W}.
%This simply means that when taking homology we do not divide by the image of the differential but by its closure. 

The $\mathbb{Z}/2$-graded $\Lambda$-vector space $$SH^{red}_M(K;\Lambda)$$ is well-defined and independent of the choice of acceleration data of $K$, see \cite{Gro}. There is a natural map 
\begin{align}\label{eq:SHSHred}
SH_M(K;\Lambda)\to SH^{red}_M(K;\Lambda),
\end{align}
 which is an isomorphism if and only if $\im(\delta)=\overline{\im}(\delta)$. The kernel is precisely the subspace of elements with infinite valuation.
 
 Let us first give a proof of Proposition \ref{prop- red-vanishing} with all these notions at hand.
 
 \begin{proof}[Proof of Proposition \ref{prop- red-vanishing}]
$SH^{red}_M(K;\Lambda)=0$ implies that the unit in $SH_M(K;\Lambda)$ has infinite valuation. Using the chain level product structure from Theorem \ref{thm-AGV} and our basic algebraic lemma (Lemma \ref{l: algebaric lemma}), we obtain that the unit must be zero. This finishes the proof.
\end{proof}

\begin{remark}\label{r:grading}
It is instructive at this point to remember that  $tel(CF(H_{n};\Lambda))$ is naturally $\mathbb{Z}/2N$-graded with $N$ being the minimal Chern number of $M$. In this case, we can take the $\mathbb{Z}/2N$-graded completion (which might be different than just completing if $N=0$) and both $SH_M(K;\Lambda)$ and $SH^{red}_M(K;\Lambda)$ are still $\mathbb{Z}/2N$-graded. Proposition \ref{prop- red-vanishing} (and its proof) works in this version as well.
% - so does many of our other results thus far but we omit spelling this out.
\end{remark}

The following theorem is the analogue of Theorem \ref{t:SHsuperheavy} with relative symplectic cohomology being replaced by the reduced symplectic cohomology.

\begin{theorem}\label{t:SHred_superheavy}
	Let $(M, \omega)$ be a closed symplectic manifold such that \eqref{eq:CYPoincare} is satisfied. For a compact subset $K$ of $M$, if the image of $PSS([vol]) \in SH_M(K;\Lambda)$ under \eqref{eq:SHSHred} in $SH^{red}_{M}(K; \Lambda)$ is not zero, then $K$ is superheavy.
\end{theorem}

\begin{proof}
The proof is almost exactly the same as the proof of Proposition \ref{p:PSS=0} and Theorem \ref{t:SHsuperheavy}.

First of all, if $c(a;nH)\to \infty$ as $n\to \infty$, then we can try to construct a primitive of a  chain level representative of $PSS_K(a)$ as in Proposition \ref{p:PSS=0}.
Using the notation from Proposition \ref{p:PSS=0}, it is clear that $\delta(y_{1}, y'_{1}, y_{2}, y'_{2}, \cdots y_{k}, y'_{k}, 0,0, \dots)$ is converging to $x$ as $k$ goes to infinity.
As a result, $x \in \overline{\im}(\delta)$
\end{proof}

The following lemma gives a sufficient condition for $SH_M(K;\Lambda) \simeq SH^{red}_M(K;\Lambda)$.

\begin{lemma}\label{l:SH=SHred}
If $\widehat{tel}(CF(H_{n}))$ has finite boundary depth, then $\im(\delta)=\overline{\im}(\delta)$ and hence $SH_M(K;\Lambda) \simeq SH^{red}_M(K;\Lambda)$.
\end{lemma}

\begin{proof}
Let $x \in \overline{\im}(\delta)$ and $x_n \in \im(\delta)$ be a sequence such that $\lim_n x_n=x$ and hence $\lim_n \val(x_{n+1}-x_n)= \infty$.
Let $\beta$ be the boundary depth of $\widehat{tel}(CF(H_{n};\Lambda)))$ and $y_n$ be such that $\delta y_n=x_{n+1}-x_n$ and $\val(y_n) \ge \val(x_{n+1}-x_n)-2\beta$.
Clearly, $y:=\sum_{n=1}^{\infty} y_n$ is a convergent sum such that $\delta y=x$. Therefore, $x \in \im(\delta)$ as well.

\end{proof}

The following lemma explains that the boundary depth of $\widehat{tel}(CF(H_{n};\Lambda)))$ does not depend on the choice of acceleration data of $K$. Because of this, we also call the boundary depth of $\widehat{tel}(CF(H_{n};\Lambda)))$ the boundary depth of $K$. 

\begin{lemma}
The boundary depth $\beta$ of $\widehat{tel}(CF(H_{n};\Lambda)))$ equals to the maximal torsion exponent $\tau$ of $SH_M(K;\Lambda_{\ge 0})$, which is independent of the choices of acceleration data.
\end{lemma}

\begin{proof}
Let $x,y \in \widehat{tel}(CF(H_{n};\Lambda)))$ be such that $\delta y=x$.  
Let $x':=T^{-\val(x)}x$ and $y':=T^{-\val(y)}y$. Both of them lie in the image of the inclusion $\widehat{tel}(CF(H_{n};\Lambda_{\ge 0}))) \to \widehat{tel}(CF(H_{n};\Lambda)))$ so we can regard them as elements in $\widehat{tel}(CF(H_{n};\Lambda_{\ge 0})))$.
We clearly have $\delta y'=T^{\val(x)-\val(y)} x'$ so $T^{\val(x)-\val(y)}[x']=0$ in $SH_M(K;\Lambda_{\ge 0})$.
By the definition of $\beta$, for fixed $x$ and $\epsilon>0$, we can find a $y$ such that $\val(x)-\val(y)< \beta +\epsilon$.
It therefore implies that $\tau \le \beta$.

Conversely, by the definition of $\tau$, for any $\epsilon>0$, there is $x' \in \widehat{tel}(CF(H_{n};\Lambda_{\ge 0})))$ such that $T^{\tau-\epsilon}[x'] \neq 0$ but $T^{\tau}[x'] = 0$.
Let $x \in \widehat{tel}(CF(H_{n};\Lambda)))$ be the image of $x'$.
The fact that $T^{\tau-\epsilon}[x'] \neq 0$ implies that there is no $y \in \widehat{tel}(CF(H_{n};\Lambda)))$ such that $\delta y=x$ and $\val(x)-\val(y)\le \tau-\epsilon$.
As a result, it proves that $\beta \le \tau$.
\end{proof}

In the $\mathbb{Z}$-graded context, it is useful to talk about having finite boundary depth in a fixed degree. Note it is possible to have finite boundary depth in all degrees, without having finite boundary depth as defined above.

\begin{corollary}\label{co: superheavy2}
	Let $(M, \omega)$ be a closed symplectic manifold with $c_{1}(TM)|_{\pi_2(M)}=0$. Let $K$ be a compact subset of $M$. Assume that \begin{itemize}
	\item $K$ is SH-full and 
	\item there is a sequence of compact sets $K_1\supset K_2\supset ...$ with finite boundary depth all containing $K$ in their interior such that $\bigcap K_i=K$.\end{itemize} Then $K$ is superheavy. 
\end{corollary}

\begin{proof}

Suppose that $K$ is $SH$-full.
Then by the Mayer-Vietoris property, for any compact set $K'$ containing $K$ in its interior, we know that 
$QH(M,\Lambda) \to SH_M(K',\Lambda)$
is an isomorphism.

As an immediate consequence of Lemma \ref{l: cy}, Theorem \ref{t:SHred_superheavy} and Lemma \ref{l:SH=SHred}., we know that $K_n$ is superheavy for all $n\geq 1$.
We conclude that $K$ is also superheavy using Lemma \ref{l:opennbhd}.
\end{proof}

Recall from Remark \ref{r:grading} that we can take $\mathbb{Z}$-graded completion when $c_{1}(TM)=0$ and the resulting $SH_M(K;\Lambda)$ is $\mathbb{Z}$-graded. We can then replace the finite boundary depth assumption by the weaker assumption which says that the boundary depth is finite at all degrees. This makes Corollary \ref{co: superheavy2} much more useful, see Theorem \ref{t:superheavyc1}. The proof of Theorem \ref{t:superheavyc1} is complete now with these remarks.

%To compare Theorem \ref{t:SHsuperheavy} and Corollary \ref{co: superheavy2}, readers should be aware that the finite boundary depth condition on $K$ only depends on $K$, while the uniform upper bound on boundary depths of $CF(nH;\Lambda)$ might depend on $H$.\footnote{Maybe we can elaborate this a bit. One implies the other?}

\bibliographystyle{amsplain}

\begin{thebibliography}{10}
	
	\bibitem {AGV} M.Abouzaid, Y.Groman and U.Varolgunes \textit{Framed $E^{2}$ structures in Floer theory.} arXiv:2210.11027.
	
	\bibitem {AS} M.Abouzaid, P.Seidel \textit{An open string analogue of Viterbo functoriality.} Geom. Topol. 14 (2010), no. 2, 627–718.
	
	\bibitem {BSV} S.Borman, N.Sheridan and U.Varolgunes \textit{Quantum cohomology as a deformation of symplectic cohomology.} Journal of Fixed Point Theory and Applications volume 24, Article number: 48 (2022).
	
	\bibitem {CFHW} K.Cieliebak, A.Floer, H.Hofer and K.Wysocki \textit{Applications of symplectic homology II: Stability of the action spectrum.} Mathematische Zeitschrift, September 1996, Volume 223, Issue 1, pp 27-45.
	
	\bibitem {DGPZ} A.Dickstein, Y.Ganor, L.Polterovich and F.Zapolsky \textit{Symplectic topology and ideal-valued measures.} arXiv:2107.10012.
	
	\bibitem {Entov} M.Entov \textit{Quasi-morphisms and quasi-states in symplectic topology.} arXiv:1404.6408 (2014).
	
	\bibitem {EP03} M.Entov and L.Polterovich \textit{Calabi quasimorphism and quantum homology.} Int. Math. Res. Not. 2003, no. 30, 1635–1676.
	
	\bibitem {EP06} M.Entov and L.Polterovich \textit{Quasi-states and symplectic intersections.} Comment. Math. Helv. 81 (2006), no. 1, 75–99.
	
	\bibitem {EP09} M.Entov and L.Polterovich \textit{Rigid subsets of symplectic manifolds.} Compos. Math. 145 (2009), no. 3, 773–826.
	
	\bibitem {FOOO} K.Fukaya, Y.-G.Oh, H.Ohta, K.Ono \textit{Spectral invariants with bulk, quasi-morphisms and Lagrangian Floer theory.} Mem. Amer. Math. Soc. 260 (2019), no. 1254, x+266 pp.
	
	\bibitem {Gro} Y.Groman \textit{Floer theory and reduced cohomology on open manifolds.} arXiv:1510.04265, To appear in Geometry \& Topology 
	
	\bibitem {HS} H.Hofer and D.Salamon \textit{Floer homology and Novikov rings.} The Floer memorial volume, 483-524, Progr. Math., 133, Birkhäuser, Basel, 1995.

	\bibitem {HLS} V.Humili\`ere, F.Le Roux and S.Seyfaddini \textit{Towards a dynamical interpretation of {H}amiltonian spectral invariants on surfaces} Geom. Topol. 20 (2016), no. 4, 2253--2334.
	
	
	\bibitem {Ish} S.Ishikawa \textit{Spectral invariants of distance functions.} J. Topol. Anal. 8 (2016), no. 4, 655–676.
	
	\bibitem {LZ} R.Leclercq and F.Zapolsky \textit{Spectral invariants for monotone Lagrangians.} J. Topol. Anal. 10 (2018), no. 3, 627–700.
	
	\bibitem {M2020} M.McLean \textit{Birational Calabi-Yau manifolds have the same small quantum products.} Ann. of Math. (2) 191 (2020), no. 2, 439-579.
	
	\bibitem {Mor} K.Morimichi \textit{Superheavy Lagrangian immersions in surfaces.} J. Symplectic Geom. 17 (2019), no. 1, 239–249.
	
	\bibitem {Oh} Y.-G.Oh \textit{Construction of spectral invariants of Hamiltonian paths on closed symplectic manifolds.} The breadth of symplectic and Poisson geometry, 525–570, Progr. Math., 232, Birkhäuser Boston, Boston, MA, 2005.
	
	\bibitem {Pardon} J.Pardon \textit{Contact homology and virtual fundamental cycles.} J. Amer. Math. Soc. 32 (2019), no. 3, 825–919.
	
\bibitem {Pol} L.Polterovich \textit{Symplectic rigidity and quantum mechanics} Eur. Math. Soc., 155-179, European Congress of Mathematics, Zürich, 2018.

	\bibitem {PSS} S.Piunikhin, D.Salamon and M.Schwarz \textit{Symplectic Floer-Donaldson theory and quantum cohomology.} Contact and symplectic geometry (Cambridge, 1994), 171–200, Publ. Newton Inst., 8, Cambridge Univ. Press, Cambridge, 1996.
	
\bibitem {S} Y.Sun \textit{Heavy sets and index bounded relative symplectic cohomology.} arXiv:2206.00066.


\bibitem {T} S.Tanny \textit{A max inequality for spectral invariants of disjointly supported Hamiltonians.} arXiv:2102.07487.

	\bibitem {TVar} D.Tonkonog and U.Varolgunes \textit{Super-rigidity of certain skeleta using relative symplectic cohomology.} Journal of Topology and Analysis, to appear.
	
	\bibitem {U} M.Usher \textit{Spectral numbers in Floer theories.} Compos. Math. 144 (2008), no. 6, 1581–1592.
	
\bibitem {U2} M.Usher \textit{Boundary depth in Hamiltonian Floer theory and its applications to Hamiltonian dynamics and coisotropic submanifolds.} Israel J. Math. 184 (2011), 1-57.
	
	\bibitem {Var} U.Varolgunes \textit{Mayer-Vietoris property for relative symplectic cohomology.} Thesis version on author's website.
	
	\bibitem {Var2021} U.Varolgunes \textit{Mayer–Vietoris property for relative symplectic cohomology.} Geometry and Topology 25 (2021) 547-642.
	
	\bibitem {Vit} C.Viterbo \textit{Functors and Computations in Floer Homology with Applications, I.} Geom. Funct. Anal. 9 (1999), no. 5, 985-1033.
	
\bibitem {W} L.Wolfgang \textit{L2-{I}nvariants: {T}heory and {A}pplications to {G}eometry and {K}-{T}heory.} Berlin, Heidelberg: Springer Berlin Heidelberg, 2002.

\end{thebibliography}

{\small

\medskip
\noindent Cheuk Yu Mak\\
\noindent School of Mathematical Sciences, University of Southampton, Southampton, SO17 1BJ, UK\\
{\it e-mail:} C.Y.Mak@soton.ac.uk

\medskip
 \noindent Yuhan Sun\\
\noindent Hill Center, Department of Mathematics, Rutgers University, 110 Frelinghuysen Rd, Piscataway NJ 08854, US\\
 {\it e-mail:} sun.yuhan@rutgers.edu

 \medskip
 \noindent Umut Varolgunes\\
\noindent Department of Mathematics, Bo\u{g}azi\c{c}i University, 34342 Bebek, Istanbul, Turkey\\
{\it e-mail:} varolgunesu@gmail.com

}

\end{document}